\documentclass[11pt]{amsart}
\usepackage{graphicx}
\usepackage{euscript}
\usepackage{amsmath}
\usepackage{amsfonts}
\usepackage{amssymb}
\usepackage{amsthm}
\usepackage{epsfig}
\usepackage{epstopdf}
\usepackage{url}
\usepackage{array}
\usepackage{amssymb}\usepackage{amscd}
\usepackage{epic}
\usepackage{eufrak}
\usepackage{tikz,underscore}
\usepackage{tikz-cd}
\usepackage{bm}
\usepackage{float}

\numberwithin{equation}{section}

\theoremstyle{plain}
\newtheorem{theorem}{Theorem}[section]
\newtheorem{lemma}[theorem]{Lemma}
\newtheorem{proposition}[theorem]{Proposition}

\newtheorem{corollary}[theorem]{Corollary}

\theoremstyle{definition}
\newtheorem{example}[theorem]{Example}
\newtheorem{remark}[theorem]{Remark}
\newtheorem{definition}[theorem]{Definition}

\def\deg{\textup{deg}}
\def\Z{\mathbb{Z}}

\def\L{\mathcal{L}}
\def\H{\mathbb{H}}
\def\F{\mathbb{F}}

\def\s{\mathfrak{s}}

\def\P{\widehat{\Phi}}
\def\A{\mathfrak{A}}
\def\bH{\mathcal{H}}
\def\C{\mathcal{C}}
\def\OS{Ozsv\'ath }

\title{L--space surgeries on 2-component L--space links}

\author{Beibei Liu}
\address{B.L.: Department of Mathematics, UC Davis, One Shields Avenue, Davis CA 95616, USA}
\email{bxliu@math.ucdavis.edu}

\begin{document}

\begin{abstract}
In this paper, we analyze L-space surgeries on two component L--space links. We show that if one surgery coefficient is negative for the L--space surgery, then the corresponding link component is an unknot. If the link  admits  very negative (i.e. $d_{1}, d_{2}\ll0$) L--space surgeries, it is the Hopf link. We also give a way to characterize  the torus link $T(2, 2l)$ by observing an L--space surgery $S^{3}_{d_{1}, d_{2}}(\L)$ with $d_{1}d_{2}<0$ on a 2-component L--space link with unknotted components. For some 2-component L--space links, we give  explicit descriptions of the  L--space surgery sets. 
\end{abstract}

\maketitle
\section{introduction}

Heegaard Floer homology is an invariant for closed, oriented 3-manifolds, defined using Heegaard diagram by \OS and Szab\'o \cite{OS2}. From the viewpoint of this invariant, \emph{L-spaces} are the simplest three-manifolds. An \emph{L-space} is a rational homology sphere such that the free rank of its Heegaard Floer homology equals the order of its first singular homology group. Boyer, Gordon and Watson recently conjectured that for closed, oriented and prime three-manifolds, left-orderability of the fundamental  groups indicates that the manifold is not an L--space \cite{BGW, HRRW, HRW, Ras17}, and this was confirmed for graph manifolds. \OS and Szab\'o proved that if the three-manifold $M$ admits an cooriented taut foliation, it is not an L--space \cite{OS3}.

A link in $S^{3}$ is an \emph{L--space link} if all sufficiently large surgeries on all components of the link are L--spaces. This indicates that $S^{3}_{d_{1}, \cdots, d_{n}}(\L)$ are L--spaces for the L--space link $\L=L_{1}\cup \cdots \cup L_{n}$ with $d_{i}\gg 0$ for all $i$. It is very hard to determine whether $S^{3}_{d_{1}, \cdots, d_{n}}(\L)$
 is an L--space for other integral surgery coefficients $d_{1}, \cdots d_{n}$. For    L--space knots $K$  in $S^{3}$, 
the integral surgery $S^{3}_{d}(K)$ is an L--space if and only if $d\geq 2g(K)-1$ \cite[Proposition 9.6]{OS11}. For multiple-component links, Manolescu and \OS constructed the truncated surgery complex
\cite{MO}. It starts with an infinitely generated complex. There are six ways to truncate this complex to a finitely generated but rather complicated complex depending on the signs of the surgery coefficients and the determinant of the surgery matrix. Y. Liu described the truncated surgery complex very explicitly for 2-component L--space links in \cite{Liu}. It is simpler compared to the truncated surgery complex for general 2-component links,  and it is possible to determine if a single surgery on the link is an L--space. However, the characterization of integral or rational L--space surgeries on 2-component L--space links is still not well-understood. 

Gorsky and N\'emethi proved that the set of L--space surgeries for most algebraic links is bounded from below and determined this set for integral surgeries along torus links \cite{GN}. Rasmussen has shown that certain torus links, satellites by algebraic links, and iterated satellites by torus links have fractal-like regions of rational L--space surgery slopes \cite{Ras}.

In this paper, we analyze the integral surgeries  $S^{3}_{d_{1}, d_{2}}(\L)$ for any 2-component L--space link $\L=L_{1}\cup L_{2}$. Note that whether a link $\L$ is an L--space link does not depend on the orientation of $\L$. However, Manolescu-\OS surgery complex depends on the orientation of $\L$. In this paper, we orient all 2-component L--space links such that they have nonnegative linking numbers. For such links, we have  the $H$-function $H_{\L}(\bm{s})$ which  is a link invariant   defined on some 2-dimensional lattice $\H(\L)$ and takes values in  nonnegative integers, see Section \ref{sec:the H-function}. If there exists a lattice point $\bm{s}=(s_{1}, s_{2})\in \H(\L)$  such that $H_{\L}(\bm{s})>H_{\L}(s_{1}, s_{2}+1), H_{\L}(\bm{s})>H_{\L}(s_{1}+1, s_{2})$ and one of  $ H_{\L}(s_{1}, \infty), H_{\L}(\infty, s_{2})$ equals $0$, we say the link is a type (A) link. Otherwise, we say the link is type (B). For example, the Whitehead link is type (A), and all algebraic links are type (B) \cite{GN}.

\begin{theorem}
\label{ppositive}

Suppose that $\L=L_{1}\cup L_{2}$ is a type (A) L--space link .  If $S^{3}_{d_{1}, d_{2}}(\L)$ is an L--space, then $d_{1}>0, d_{2}>0$ and $\det \Lambda>0$. 

\end{theorem}

For type (A) L--space links, the region for  L--space surgeries is bounded from below. For type (B) L--space links, the region for  L--space surgeries is more complicated, and it may be unbounded. For example,  the torus link $T(4, 6)$ has unbounded L--space surgery set (see  \cite[Figure 1]{GN}). We analyze the very positive or very negative surgeries $S^{3}_{d_{1}, d_{2}}(\L)$ where $d_{i}\gg 0$ or $d_{i}\ll 0$ for type (B) links. 

\begin{theorem}
\label{thm:unknot}
Let $\L=L_{1}\cup L_{2}$ be a type (B) L--space link. If $S^{3}_{d_{1}, d_{2}}(\L)$ is an L--space with $d_{1}\gg 0, d_{2}\ll 0$, then $L_{2}$ is an unknot.  

\end{theorem}

For algebraic links, Gorsky and N\'emethi proved a stronger result by a very different method. 

\begin{theorem}\cite[Theorem 1.3.2]{GN}
Suppose that $\L=L_{1}\cup L_{2}$ is an algebraic link with two components. Then $S^{3}_{d_{1}, d_{2}}(\L)$ is an L--space for $d_{1}\gg 0$ and $d_{2}\ll 0$ if and only if $L_{2}$ is an unknot. 
\end{theorem}

\begin{theorem}
\label{thm:hopf}
Let $\L=L_{1}\cup L_{2}$ be a nontrivial 2-component L--space link. If $S^{3}_{d_{1}, d_{2}}(\L)$ is an L--space for $d_{1}\ll 0$ and $ d_{2}\ll 0$, then $\L$ is the Hopf link. 
\end{theorem}

Next, we consider 2-component L--space links with unknotted components. 

\begin{theorem}
\label{thm:torus}
Let $\L=L_{1}\cup L_{2}$ be an L--space link with unknotted components and linking number $l$. If $S^{3}_{d_{1}, d_{2}}(\L)$ is an L--space for $d_{1}d_{2}<0$, then $L$ is the torus link $T(2, 2l)$. 
\end{theorem}

This gives a characterization of the torus link $T(2, 2l)$.

For type (A) L--space links $\L$, the region for possible L--space surgeries is in the first quadrant, but it is still not clear which surgery is an L--space. We consider 2-component L--space links with vanishing linking numbers. 

Based on the $H$-function of the link $\L$, we define:
\begin{equation}
\label{defineb}
b_{1}(\L)=\min \{\lceil s_{1}-1 \rceil \mid H(s_{1}, s_{2})=H(\infty, s_{2}) \textup{ for all } s_{2} \}.
\end{equation}
$b_{2}(\L)$ is similarly defined.  If the context is clear, we write $b_{i}(\L)$ as $b_{i}$ for simplicity. 

\begin{theorem}\cite[Theorem 5.1]{GLM}
Assume that $\L$ is a nontrivial L--space link with unknotted components and linking number zero. Then $S^{3}_{d_{1}, d_{2}}(\L)$ is an L--space if and only if $d_{1}>2b_{1}$ and $d_{2}>2b_{2}$. 
\end{theorem} 

For general 2-component L--space links with vanishing linking numbers, $b_{i}\geq g(L_{i})-1$ for $i=1, 2$. We indicate the  possible L--space regions of such links as follows. 

\begin{theorem}
Let $\L=L_{1}\cup L_{2}$ be an oriented L--space link with  linking number zero. Suppose that $b_{i}=g_{i}-1$ for $i=1, 2$ where $g_{i} $ is the genus of the knot $L_{i}$. Then $S^{3}_{d_{1}, d_{2}}(\L)$ is an L--space if and only if $d_{1}>2b_{1}$ and $d_{2}>2b_{2}$. 

\end{theorem}

\begin{theorem}
Let $\L=L_{1}\cup L_{2}$ be an oriented nontrivial L--space link with  linking number zero.  Suppose that $b_{i}\geq g_{i}$ for $i=1$ or $2$. The possible L--space surgeries are indicated in Figure \ref{F11}. The green color indicated the region of L--space surgeries. Points in the red regions won't give $L$--space surgeries, and the white regions are points for possible L--space surgeries. 

\end{theorem}

\begin{figure}[H]
\centering
\includegraphics[width=6.0in]{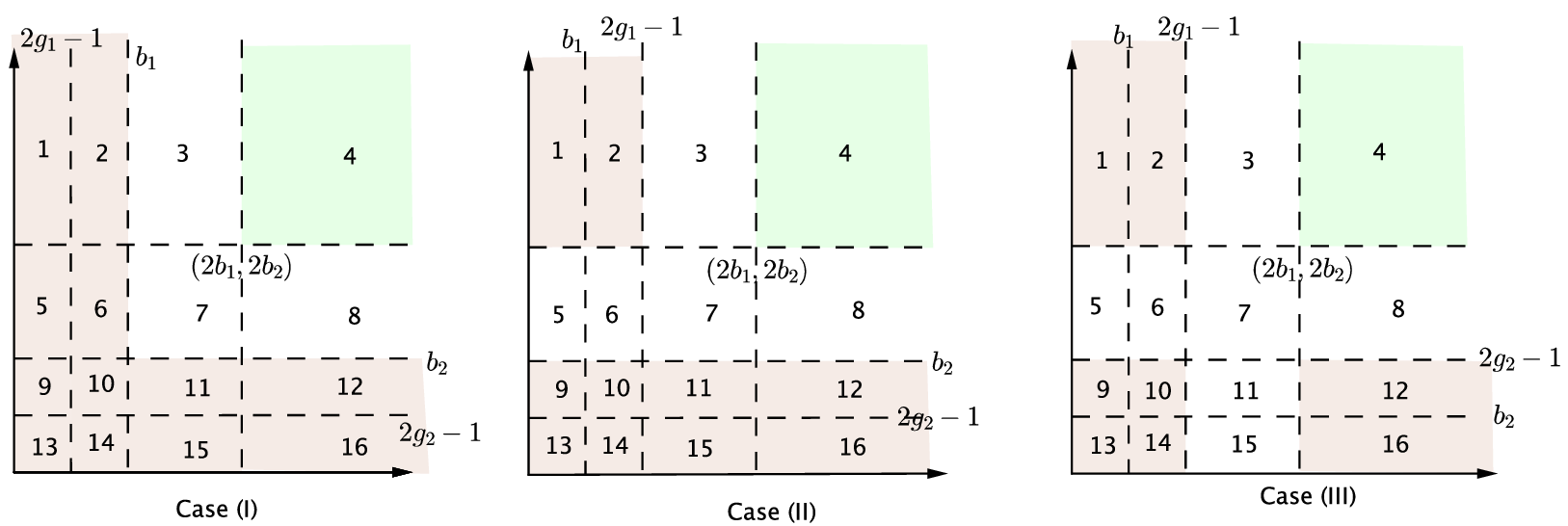}
\caption{L-space surgeries for $\L$ \label{F11}}
\end{figure}

If the  symmetrized Alexander polynomial of the 2-component L--space  link $\L$ satisfies some additional properties, we have a more precise description of the set of L--space surgeries.

A lattice point $\bm{s}=(s_{1}, s_{2})\in \H(\L)$ is called \emph{maximal} if $H_{\L}(\bm{s})=1, H_{\L}(s_{1}+1, s_{2})=H_{\L}(s_{1}, s_{2}+1)=0$.

\begin{theorem}
\label{thm:if and only}
Let $\L=L_{1}\cup L_{2}$ be an L--space link with $b_{1}=s_{1}$  and $b_{2}=s'_{2}$ for maximal lattice points $(s_{1}, s_{2})$ and  $(s'_{1}, s'_{2})$.  Suppose that the coefficients of $t_{1}^{-s_{1}-1/2}t_{2}^{s_{2}+1/2}$ and $t_{1}^{s'_{1}+1/2}t_{2}^{-s'_{2}-1/2}$ in the symmetrized Alexander polynomial $\Delta_{\L}(t_{1}, t_{2})$ are nonzero. Then $S^{3}_{d_{1}, d_{2}}(\L)$ is an L--space if and only if $d_{1}>2b_{1}$ and $d_{2}>2b_{2}$. 
\end{theorem}

For example, the Whitehead link satisfies the assumption in Theorem \ref{thm:if and only}. Furthermore, if the link $\L=L_{1}\cup L_{2}$ satisfies the assumption in Theorem \ref{thm:if and only}, the cable link $\L_{p, q}=L_{(p, q)}\cup L_{2}$ also satisfies the assumption where $p, q$ are coprime positive integers with $q/p$ sufficiently large and $L_{(p, q)}$ denotes the $(p, q)$-cable of $L_{1}$. Note that $\L_{p,q}$ is also an L--space link if $\L$ is an L--space link \cite{BG}.

\begin{theorem}
\label{thm:cable}
Suppose that $\L$ is an L--space link that satisfies the assumption in  Theorem \ref{thm:if and only}. Then $S^{3}_{d_{1}, d_{2}}(\L_{p, q})$ is an L--space if and only if $d_{1}>2b_{1}(\L_{p, q})$ and $d_{2}>2b_{2}(\L_{p, q})$ .
\end{theorem}

\begin{remark}
The constants $b_{1}(\L_{p, q})$ and $b_{2}(\L_{p, q})$ can be obtained from $b_{1}(\L)$ and $b_{2}(\L)$, see Lemma \ref{tranb}. 

\end{remark}

\begin{corollary}
\label{whitehead}
Let $Wh=L_{1}\cup L_{2}$ denote the Whitehead link and $Wh_{cab}=L_{(p_{1}, q_{1})}\cup L_{(p_{2}, q_{2})}$ be the cable link where $L_{(p_{i}, q_{i})}$ is the $(p_{i}, q_{i})$-cable of $L_{i}$ and $p_{i}, q_{i}$ are coprime positive integers with $q_{i}/p_{i}$ sufficiently large. The surgery manifold $S^{3}_{d_{1}, d_{2}}(Wh_{cab})$ is an $L$-space if and only if $d_i\geq p_{i}q_{i}+p_{i}-q_{i}-1$ for $i=1, 2$.

\end{corollary}

The main ingredient of the proofs is Manolescu-\OS truncated surgery complex. For a 2-component L--space link $\L=L_{1}\cup L_{2}$,  the subcomplexes $\A^{00}_{\bm{s}}, \A^{01}_{\bm{s}}, \A^{10}_{\bm{s}}$ and $\A^{11}_{\bm{s}}$ are used to keep track of the  filtration information induced by the link $\L$, its sublinks $L_{1}, L_{2}$ and $\emptyset$. We construct a CW-complex corresponding to this truncated surgery complex. More precisely, we associate a 2-dimensional cell to  $\A^{00}_{\bm{s}}$, a 1-dimensional cell to $\A^{01}_{\bm{s}}$ or $\A^{10}_{\bm{s}}$ and a 0-dimensional cell to $\A^{11}_{\bm{s}}$. The singular homology of the CW complex corresponds to the generators of the free part of $HF^{-}(S^{3}_{d_{1}, d_{2}}(\L), \mathfrak{s})$. If the surgery $S^{3}_{d_{1}, d_{2}}(\L)$ is an L--space, we should be able to locate the generator of $HF^{-}(S^{3}_{d_{1}, d_{2}}(\L), \mathfrak{s})$. For example, if $d_{1}\ll 0, d_{2}\ll 0$, the CW complex corresponding to the truncated surgery complex in each Spin$^{c}$ structure is a square. It is contractible, so its singular homology is generated by a class of a $0$-cell. Then the generator of $HF^{-}(S^{3}_{d_{1}, d_{2}}(\L), \mathfrak{s})$ is $H_{\ast}(\A^{11}_{\bm{s}})$ for some $\bm{s}\in \H(\L)$. For details, see Section \ref{sec:truncation} and \cite{GLM}. This will give restrictions to the differentials in the surgery complex, which is related to the $H$-function.

\medskip
{\bf Organization of the paper.} In Section \ref{sec:the H-function}, we give the definition and properties of  the $H$-function for oriented links. In  Section \ref{sec:L-space links}, we give  the definition and properties of L--space links and give a way to compute the $H$-function of L--space links in terms of their Alexander polynomials. In Section \ref{sec:truncation}, we review  the truncated surgery complex introduced by Manolescu and \OS \cite{MO} and associate to it  a CW complex. In Section \ref{typea}, we discuss type (A) L--space links and prove Theorem \ref{ppositive}. In Section \ref{typeb}, we discuss type (B) L--space links and prove Theorem \ref{thm:unknot}, Theorem \ref{thm:hopf} and Theorem \ref{thm:torus}. In Section \ref{sec:linkingnumber0}, we discuss L--space links with vanishing linking numbers and describe the possible L--space surgery sets. In Section \ref{expdes}, we give explicit descriptions of L--space surgery sets for some 2-component L--space links, and prove Theorem \ref{thm:if and only}, Theorem \ref{thm:cable} and Corollary \ref{whitehead}.

\medskip
{\bf Notation and Conventions.}
In this paper, all 2-component links are oriented such that the linking number is nonnegative. 
We use  $l$ to  denote the linking number and  $\Lambda$ to denote the surgery matrix 
$$
\Lambda=
\begin{pmatrix}
d_{1} & l \\
l       & d_{2}
\end{pmatrix}.
$$
We use $\L$ to denote links in $S^{3}$ and $L_{1}, \cdots, L_{n}$ to denote the link components in the same link. We denote vectors in $\mathbb{R}^{n}$ by bold letters. For two vectors $\bm{u}=(u_{1}, \cdots, u_{n})$ and $\bm{v}=(v_{1}, \cdots, v_{n})$, we write $\bm{u}\succeq \bm{v}$ if $u_{i}\geq v_{i}$ for all $1\leq i \leq n$, and $\bm{u}\succ \bm{v}$ if $\bm{u}\succeq \bm{v}$ and $\bm{u}\neq \bm{v}$. Let $\bm{e}_{i}$ denote a vector in $\mathbb{R}^{n}$ where the $i$-th entry is $1$ and other entries are $0$. For a subset $B\subset \{1, \cdots, n\}$, let $\bm{e}_{B}=\sum_{i\in B} \bm{e}_{i}$.  Let $\Delta_{\L}(t_{1}, \cdots, t_{n})$ denote the symmetrized Alexander polynomial of $\L$. Throughout this paper, we work over the field $\F=\Z/2\Z$. 

\medskip
{\bf Acknowledgements.}
I deeply appreciate Eugene Gorsky for his support and useful  suggestions  during this project. I am grateful to Allison Moore for helpful discussions, and I also want to thank Jacob Rasmussen for pointing out that the Thurston polytope characterizes the torus link $T(2, 2l)$.  The project is partially supported by NSF grant DMS-1700814.

\section{The H-function and L--space links}

\subsection{The H-function}
\label{sec:the H-function}
\OS  and Szab\'o associated  chain complexes  $CF^{-}(M), \widehat{CF}(M)$ to an admissible Heegaard diagram for a closed oriented connected 3-manifold $M$ \cite{OS1}, and these give  three-manifold invariants $HF^{-}(M)$ and $\widehat{HF}(M)$. A nullhomologous link $\L=L_{1}\cup \cdots \cup L_{n}$ in $M$ defines a filtration on the chain complex $CF^{-}(M)$. For links in $S^{3}$, this filtration is indexed by an $n$-dimensional lattice $\H(\L)$ which is defined as follows:

\begin{definition}
For an oriented link $\L=L_{1}\cup \cdots \cup L_{n}\subset S^{3}$, define $\H(\L)$ to be the affine lattice over $\Z^{n}$:
$$\H(\L)=\bigoplus_{i=1}^{n} \H_{i}(\L), \quad \H_{i}(\L)=\Z+\dfrac{lk(L_{i}, \L\setminus L_{i})}{2}$$
where $lk(L_{i}, \L\setminus L_{i})$ denotes the linking number of $L_{i}$ and $\L\setminus L_{i}$. 
\end{definition}

Given $\bm{s}=(s_{1}, \cdots, s_{n})\in \H(\L)$, the \emph{generalized Heegaard Floer complex} $\A^{-}(\L, \bm{s})\subset CF^{-}(S^{3})$ is the $\F[[U]]$-module defined to be a subcomplex of $CF^{-}(S^{3})$ corresponding to the filtration indexed by $\bm{s}$ \cite{MO}. The link Floer homology $HFL^{-}(\L, \bm{s})$ is the homology of the associated complex with respect to this filtration, and is a module over $\F[[U]]$. For more details, see \cite{ BG, MO}. 

By the large surgery theorem \cite[Theorem 12.1]{MO}, the homology of $\A^{-}(\L, \bm{s})$ is isomorphic to the Heegaard Floer homology of a large surgery on the link $\L$ equipped with some Spin$^{c}$ structure as a $\F[[U]]$-module. Thus the homology of $\A^{-}(\L, \bm{s})$ is a direct sum  of one copy of $\F[[U]]$ and some $U$-torsion submodule. 

\begin{definition}\cite[Definition 3.9]{BG}
For an oriented link $\L\subset S^{3}$, we define the $H$-function $H_{\L}(\bm{s})$ by saying that  $-2H_{\L}(\bm{s})$ is the maximal homological degree of the free part of $H_{\ast}(\A^{-}(\L, \bm{s}))$ where $\bm{s}\in \H(\L)$.

\end{definition}

\begin{remark}
We sometimes write $H_{\L}(\bm{s})$ as $H(\bm{s})$ for simplicity if there is no confusion. 
\end{remark}

We list several properties of the $H$-function as follows.

\begin{lemma}\cite[Proposition 3.10]{BG}
\label{growth control}
For an oriented link $\L\subset S^{3}$, the $H$-function $H_{\L}(\bm{s})$ takes nonnegative values, and $H_{\L}(\bm{s}-\bm{e}_{i})=H_{\L}(\bm{s})$ or $H_{\L}(\bm{s}-\bm{e}_{i})=H_{\L}(\bm{s})+1$ where $\bm{s}\in \H(\L)$. 

\end{lemma}

\begin{lemma}\cite[Proposition 3.12]{BG}
\label{infinity}
For an oriented link $\L=L_{1}\cup \cdots \cup L_{n}\subset S^{3}$ and $\bm{s}=(s_{1}, \cdots, s_{n})\in \H(\L)$, 
$$H_{\L}(s_{1},\cdots, s_{n-1}, \infty)=H_{\L\setminus L_{n}}(s_{1}-lk(L_{1}, L_{n})/2, \cdots, s_{n-1}-lk(L_{n-1}, L_{n})/2 )$$
where $lk(L_{i}, L_{n})$ denotes the linking number of $L_{i}$ and $L_{n}$ for $i=1, 2, \cdots, n-1$.
\end{lemma}

\begin{remark}
We use the convention that $H_{\L}(\infty, \cdots, \infty)=0$.
\end{remark}

\subsection{L--space links}
\label{sec:L-space links}
In \cite{OS05}, \OS and Szab\'o introduced the concept of L--spaces. 

\begin{definition}
A 3-manifold $Y$ is an L--space if it is a rational homology sphere and its Heegaard Floer homology has minimal possible rank: for any Spin$^{c}$-structure $\s$, $\widehat{HF}(Y, \s)=\F$ and $HF^{-}(Y, \s)$ is a free $\F[U]$-module of rank 1. 
\end{definition}

\begin{definition}\cite{GN15, Liu}
An  $n$-component link $\L\subset S^{3}$ is an L--space link if there exists $\bm{0}\prec \bm{p}\in \Z^{n}$ such that the surgery manifold $S_{\bm{q}}(\L)$ is an L--space for any $\bm{q}\succeq \bm{p}$.
\end{definition}

The following properties of L--space links will be used in this paper. 

\begin{theorem}\cite{Liu}
\label{L-space link pro}
(a) Every sublink of an L--space link is an L--space link. 

(b) A link is an L-space link if for all $\bm{s}$ one has $H_{\ast}(\A^{-}(\L, \bm{s}))=\F[[U]]$. 
\end{theorem}

\begin{lemma}\cite[Lemma 2.5]{Liu}
\label{induction}
Let $\L=L_{1}\cup \cdots \cup L_{n}$ be a link with n components, and $\L'=\L-L_{1}$. Let $\Lambda$ be the framing matrix of $\L$ for the surgery $S^{3}_{d_{1}, \cdots, d_{n}}(\L)$, and denote by $\Lambda'$ the restriction of $\Lambda$ on $\L'$. Suppose $S^{3}_{d_{1}, \cdots, d_{n}}(\L)$ and $S^{3}_{d_{2}, \cdots, d_{n}}(\L')$ are both L--spaces. Then, 
\begin{enumerate}
\item Case I: if $\det(\Lambda)\cdot \det(\Lambda')>0$, then for all $k>0$, $S^{3}_{d_{1}+k, d_{2}, \cdots, d_{n}}(\L)$ is an L-space;

\item Case II: if $\det(\Lambda)\cdot \det(\Lambda')<0$, then for all $k>0$, $S^{3}_{d_{1}-k, d_{2}, \cdots, d_{n}}(\L)$ is an L--space.

\end{enumerate}

\end{lemma}

For L--space links, the $H$-function can be computed from the Alexander polynomial. One can write \cite{BG}:
\begin{equation}
\label{comh}
\chi(HFL^{-}(\L, \bm{s}))=\sum_{B\subset \{1, \cdots, n\}} (-1)^{|B|-1}H_{\L}(\bm{s}-\bm{e}_{B}). 
\end{equation}

The Euler characteristic $\chi(HFL^{-}(\L, \bm{s}))$ was computed in \cite{OS08}, 
\begin{equation}
\label{Alex}
\tilde{\Delta}(t_{1}, \cdots, t_{n})=\sum_{\bm{s}\in \H(\L)} \chi(HFL^{-}(\L, \bm{s}))t_{1}^{s_{1}}\cdots t_{n}^{s_{n}}
\end{equation}
where $\bm{s}=(s_{1}, \cdots, s_{n})$ and 
$$
\tilde{\Delta}_{\L}(t_{1}, \cdots, t_{n}): = \left\{
        \begin{array}{ll}
           (t_{1}\cdots t_{n})^{1/2} \Delta_{\L}(t_{1}, \cdots, t_{n}) & \quad \textup{if } n >1, \\
            \Delta_{\L}(t)/(1-t^{-1}) & \quad  \textup{if } n=1. 
        \end{array}
    \right. 
$$
Note that we regard $\dfrac{1}{1-t^{-1}}$ as an infinite power series in $t^{-1}$. The Alexander polynomial $\Delta_{\L}(t_{1}, \cdots, t_{n})$ is normalized so that it is symmetric about the origin. 

One can use \eqref{comh} to compute the $H$-function of $\L$ by using the values of the $H$-function for sublinks as the boundary conditions. In this paper, we mainly consider links with one and two components. 

For $n=1$, the equation $\eqref{comh}$ has the form:
\begin{equation}
\label{diff}
\chi(HFL^{-}(\L, s))=H(s-1)-H(s).
\end{equation}
It is not hard to see that if $\L$ is an unknot, $H(s)=\dfrac{s-|s|}{2}$. 

The genus of a knot $\L$ is defined as:
$$g(\L)=\min \{\textup{genus}(F)\mid F\subset S^{3} \textup{ is an oriented, embedded surface with } \partial F=\L \}.$$

\begin{lemma}
\label{genus0}
Let $\L$ be an L--space knot. Then $H(s)=0$ if and only if $s\geq g(\L)$ where  $s\in \H(\L)\cong \Z$. 

\end{lemma}

\begin{proof}
Let $\Delta_{\L}(t)=\sum_{s\in \Z} a_{s} t^{s}$ denote its symmetrized Alexander polynomial. We claim that  $g(\L)=\max\{ s \mid a_{s}\neq 0\}$. Recall that 
$$\Delta_{\L}(t)=\sum_{s\in \Z}\chi(\widehat{HFK}(\L, s))\cdot t^{s}$$
where $\widehat{HFK}(\L, s)$ is a knot invariant from the Heegaard Floer package \cite{OS04b}, and  \cite[Theorem 1.2]{OS3}
$$g(\L)=\max\{ s\mid \widehat{HFK}(\L, s)\neq 0\}.$$
Hence, $g(\L)$ is the top degree of $t$ in $\Delta_{\L}(t)$. Observe that the top degrees of $t$ in $\Delta_{\L}(t)$ and $\tilde{\Delta}_{\L}(t)$ are the same. By \eqref{diff}:
$$H(s-1)-H(s)=0, \textup{ for all } s>g, \quad H(g-1)-H(g)=a_{g}\neq 0.$$
By Lemma \ref{growth control} and the boundary condition $H(\infty)=0$, we have 
$$H(s)=0 \textup{ for all } s\geq g, \textup{ and } H(s)\geq 1 \textup{ for all } s\leq g-1.$$

\end{proof}

For  2-component L--space links $\L=L_{1}\cup L_{2}$, \eqref{comh} has the form, 
\begin{equation}
\label{computeh}
\chi(HFL^{-}(\L, \bm{s}))=-H(s_{1}-1, s_{2}-1)+H(s_{1}-1, s_{2})+H(s_{1}, s_{2}-1)-H(s_{1}, s_{2}),
\end{equation}
and we have $H(s_{1}, \infty)=H_{1}(s_{1}-l/2)$ and $H(\infty, s_{2})=H_{2}(s_{2}-l/2)$ where $H_{1}, H_{2}$ denote the $H$-functions of $L_{1}, L_{2}$ respectively, and $l$ is the linking number.

For general L--space links $\L$, the $H$-function  satisfies the following conjugation symmetry. 

\begin{lemma}\cite[Lemma 5.5]{Liu}
\label{symmetry2}
For an oriented $n$-component L--space link $\L\subset S^{3}$, the $H$-function satisfies $H_{\L}(-\bm{s})=H_{\L}(\bm{s})+\sum^{n}_{i=1} s_{i}$ where $\bm{s}=(s_{1}, \cdots, s_{n})\in \H(\L)$. 

\end{lemma}

\begin{corollary}
\label{symmetry3}
For an oriented 2-component link $\L=L_{1}\cup L_{2}\subset S^{3}$ with linking number $l$, one has
$$H(-s_{1}, -s_{2})-H(\infty, -s_{2})=H(s_{1}, s_{2})-H(\infty, s_{2}+l)+s_{1}-l/2.$$

\end{corollary}

\begin{proof}
By Lemma \ref{infinity} and Lemma \ref{symmetry2}, we have:
\begin{eqnarray*}
	H(-s_{1}, -s_{2})-H(\infty, -s_{2}) &=& H(s_{1}, s_{2})+s_{1}+s_{2}-H_{2}(-s_{2}-l/2)\\
	&=& H(s_{1}, s_{2})+s_{1}+s_{2}-H_{2}(s_{2}+l/2)-s_{2}-l/2 \\
	&=& H(s_{1}, s_{2})-H(\infty, s_{2}+l)+s_{1}-l/2. 
\end{eqnarray*}
\end{proof}

The surgeries on the  link $\L$ do not depend on its orientation, so whether a link $\L$ is an L--space link does not depend on orientations. However, the $H$-function of $\L$ depends on its orientation.
\begin{proposition}
 Let $\L=L_{1}\cup L_{2}$ be an oriented 2-component L--space link with linking number $l$, and $\L'=-L_{1}\cup L_{2}$ be the link obtained from $\L$ by reversing the orientation of $L_{1}$. Then for any $(s_{1}, s_{2})\in \H(\L')$
 $$H_{\L'}(s_{1}, s_{2})=H_{\L}(-s_{1}, s_{2})-s_{1}-l/2.$$
 \end{proposition}
 
 \begin{proof}
Since $\L$ is an L--space link, $\L'$ is also an L--space link. Let $\phi(s_{1}, s_{2})=H_{\L}(-s_{1}, s_{2})-s_{1}-l/2$. It suffices to prove that $\phi$ satisfies \eqref{computeh} and the boundary condition that $\phi(s_{1}, \infty)=H_{-L_{1}}(s_{1}+l/2)$ and $\phi(\infty, s_{2})=H_{L_{2}}(s_{2}+l/2)$.  We check the boundary condition first. Recall that the Alexander polynomial of the knot $-L_{1}$ is obtained from the Alexander polynomial of $L_{1}$ by substituting $t^{-1}$ for $t$. Then $-L_{1}$ and $L_{1}$ have the same symmetrized Alexander polynomial and both of them are L--space knots. By \eqref{diff}, $H_{-L_{1}}(s_{1})=H_{L_{1}}(s_{1})$ for all $s_{1}\in \Z$. Then 
\begin{eqnarray*}
	\phi(s_{1}, \infty) &=& H_{\L}(-s_{1}, \infty)-s_{1}-l/2\\
	&=& H_{L_{1}}(-s_{1}-l/2)-s_{1}-l/2 \\
	&=& H_{L_{1}}(s_{1}+l/2)=H_{-L_{1}}(s_{1}+l/2). 
\end{eqnarray*}
The proof of $\phi(\infty, s_{2})=H_{L_{2}}(s_{2}+l/2)$ is similar by observing that 
$$H(-s_{1}, s_{2})=H(s_{1}, -s_{2})+s_{1}-s_{2}.$$
Now we check that $\phi(s_{1}, s_{2})$ satisfies \eqref{computeh}. Note that $\Delta_{\L'}(t_{1}, t_{2})=-\Delta_{\L}(t_{1}^{-1}, t_{2})$ \cite{Kid}. Assume that $\Delta_{\L'}(t_{1}, t_{2})=\sum a_{s_{1}, s_{2}}t_{1}^{s_{2}}t_{2}^{s_{2}}$, and  $\Delta_{\L}(t_{1}, t_{2})=\sum b_{s_{1}, s_{2}}t_{1}^{s_{1}}t_{2}^{s_{2}}$.	Then $a_{s_{1}, s_{2}}=-b_{-s_{1}, s_{2}}$. By \eqref{Alex},  $\chi(HFL^{-}(\L', (s_{1}, s_{2})))=a_{s_{1}-1/2, s_{2}-1/2}$, and $\chi(HFL^{-}(\L, (s_{1}, s_{2})))=b_{s_{1}-1/2, s_{2}-1/2}$. Observe that 
\begin{align*}
   & \quad  -\phi(s_{1}-1, s_{2}-1)+\phi(s_{1}-1, s_{2})+\phi(s_{1}, s_{2}-1)-\phi(s_{1}, s_{2})\\ 
   &=-H_{\L}(-s_{1}+1, s_{2}-1)+H_{\L}(-s_{1}+1, s_{2})+H_{\L}(-s_{1}, s_{2}-1)-H_{\L}(-s_{1}, s_{2})\\
   &=-\chi(HFL^{-}(\L, (-s_{1}+1, s_{2})))=-b_{-s_{1}+1/2, s_{2}-1/2}\\
   &=a_{s_{1}-1/2, s_{2}-1/2}=\chi(HFL^{-}(\L', (s_{1}, s_{2}))). 
\end{align*}
Thus, $H_{\L'}(s_{1}, s_{2})=H_{\L}(-s_{1}, s_{2})-s_{1}-l/2.$

 \end{proof}

\section{Surgery complex and truncations }
\label{sec:truncation}
\subsection{Truncated surgery complexes for 2-component L--space links}
\label{subsec:truncation}
We first review the Manolescu-Ozsv\'ath link surgery complex \cite{MO} for oriented links $\L=L_{1}\cup L_{2}$. 

For any sublink $M\subset \L$, set $N=\L-M$. We choose an orientation on $M$ (possibly different from the one induced from $\L$), and denote the corresponding oriented link by $\vec{M}$.   One defines the map
$$\psi^{\vec{M}}: \H(\L)\rightarrow \H(N)$$
as in \cite{MO}. The map $\psi^{\vec{M}}$ depends only on the $i$-summand $\H_{i}(\L)$ of $\H(\L)$ corresponding to $L_{i}\subset N$. Each of $L_{i}$'s appears in $N$ with a  index $j_{i}$, so there is a corresponding summand $\H_{j_{i}}(N)$ of $\H(N)$. Set 
$$\psi^{\vec{M}}_{i}: \H_{i}(\L)\rightarrow \H_{j_{i}}(N), \quad s_{i}\rightarrow s_{i}-\dfrac{lk(L_{i}, \vec{M})}{2}.$$
Then define $\psi^{\vec{M}}$ to be the direct sum of the maps $\psi^{\vec{M}}_{i}$ precomposed with the relevant factors. 

For sublinks $M\subset \L$ with orientation induced from $\L$, we use $\mathcal{H}^{\L-M}$ to denote the Heegaard diagram of $\L-M$ obtained from $\mathcal{H}^{\L}$ by forgetting the $z$ basepoints on the sublink $M$. The diagram $\mathcal{H}^{L-M}$ is associated with the generalized Floer complex $\A^{-}(\mathcal{H}^{\L-M}, \psi^{M}(\bm{s}))$. 

For the general 2-component link $\L$, we describe the chain complex and its differential in detail. We write 
\[
\Lambda=\begin{pmatrix}
d_{1} & l \\
l & d_{2}
\end{pmatrix}.
\]
as the surgery matrix where $l$ denotes the linking number and $d_1, d_2$ denote the surgery coefficients. 

For a link $\L=L_{1}\cup L_{2}$, a two digit binary superscript is used to keep track of which link components are forgotten. Let $\A^{00}_{\bm{s}}=\A^{-}(\mathcal{H}^{\L}, \bm{s}), \A^{01}_{\bm{s}}=\A^{-}(\mathcal{H}^{\L-L_{2}}, s_{1}-l/2), \A^{10}_{\bm{s}}=\A^{-}(\mathcal{H}^{\L-L_{1}}, s_{2}-l/2)$ and $\A^{11}_{\bm{s}}=\A^{-}(\mathcal{H}^{L-L_{1}-L_{2}}, \emptyset)$ where $\bm{s}=(s_{1}, s_{2})\in \H(\L)$. Let 
$$\mathcal{C}_{\bm{s}}=\bigoplus_{\epsilon_{1}, \epsilon_{2}\in \{0, 1\}} \A^{\epsilon_{1}\epsilon_{2}}_{\bm{s}}.$$
The surgery complex is defined as 
$$\mathcal{C}(\mathcal{H^{\L}}, \Lambda)=\prod_{\bm{s}\in \H(\L)} \mathcal{C}_{\bm{s}}. $$ 

The differential in the complex is defined as follows. Consider sublinks $\emptyset, \pm L_{1}, \pm L_{2}$ and $\pm L_{1} \pm L_{2}$ where $\pm$ denotes whether or not the orientation of the sublink is the same as the one induced from $\L$. Based on  \cite{MO}, we have the following maps, where $\Phi^{\emptyset}_{\bm{s}}$ is the internal differential on any chain complex $\A^{\epsilon_{1}\epsilon_{2}}_{\bm{s}}$. 
\begin{eqnarray}
\label{maps1}
\begin{aligned}
\Phi^{L_{1}}_{\bm{s}}: \A^{00}_{\bm{s}}\rightarrow \A^{10}_{\bm{s}}, \quad &\Phi^{-L_{1}}_{\bm{s}}: \A^{00}_{\bm{s}}\rightarrow \A^{10}_{\bm{s}+\Lambda_{1}}, \\
\Phi^{L_{2}}_{\bm{s}}: \A^{00}_{\bm{s}}\rightarrow \A^{01}_{\bm{s}}, \quad &\Phi^{-L_{2}}_{\bm{s}}: \A^{00}_{\bm{s}}\rightarrow \A^{01}_{\bm{s}+\Lambda_{2}}, \\
\Phi^{L_{1}}_{s_{1}}: \A^{01}_{\bm{s}}\rightarrow \A^{11}_{\bm{s}}, \quad &\Phi^{-L_{1}}_{s_{1}}: \A^{01}_{\bm{s}}\rightarrow \A^{11}_{\bm{s}+\Lambda_{1}}, \\
\Phi^{L_{2}}_{s_{2}}: \A^{10}_{\bm{s}}\rightarrow \A^{11}_{\bm{s}}, \quad &\Phi^{-L_{2}}_{s_{2}}: \A^{10}_{\bm{s}}\rightarrow \A^{11}_{\bm{s}+\Lambda_{2}}, 
\end{aligned}
\end{eqnarray}
where $\Lambda_{i}$ is the $i$-th column of $\Lambda$. We did not write the maps $\Phi_{\bm{s}}^{\pm L_{1}\pm L_{2}}$ in detail since we will focus on L--space links and these maps vanish for 2-component L--space links. 
Let 
\[
	D_{\bm{s}}=\Phi^{\varnothing}_{\bm{s}}+\Phi^{\pm L_{1}}_{\bm{s}}+\Phi^{\pm L_{2}}_{\bm{s}}+\Phi^{\pm L_{1}}_{s_{1}}+\Phi^{\pm L_{2}}_{s_{2}}+\Phi^{\pm L_{1}\pm L_{2}}_{\bm{s}},
\]
and let $D=\prod_{\bm{s}\in \H(\L)} D_{\bm{s}}$. Then $(\C(\bH^{\L}, \Lambda), D)$ is the Manolescu-Ozsv\'ath surgery complex.  

\begin{lemma}\cite[Lemma 10.1]{MO}
\label{isomorphism}
There exists a constant $b\in \mathbb{N}$ such that for any $i=1, 2$, and for any sublink $M\subset L$ not containing the component $L_{i}$, the chain map 
\[
	\Phi^{\pm L_{i}}_{\psi^{\vec{M}}(\bm{s})}: \mathfrak{A}^{-}(\mathcal{H}^{L-M}, \psi^{\vec{M}}(\bm{s}))\rightarrow \A^{-}(\bH^{L-M-L_{i}}, \psi^{\vec{M}\cup \pm L_{i}}(\bm{s}))
\]
induces an isomorphism on homology provided that either 
\begin{itemize}
\item $\bm{s}\in \H(\L)$ is such that $s_{i}>b$, and $L_{i}$ is given the orientation induced from $L$; or 

\item $\bm{s}\in \H(\L)$ is such that $s_{i}<-b$, and $L_{i}$ is given the orientation opposite to the one induced from $L$. 
\end{itemize}

\end{lemma}



\subsection{Perturbed surgery formula.}
\label{perturbed}
Up to homotopy equivalence, one can replace every complex $\A^{\epsilon_{1}\epsilon_{2}}_{\bm{s}}$ where $\epsilon_{1}, \epsilon_{2}=0$ or $1$ by its chain homotopy type and replace every differential map $\Phi^{\pm L_{i}}_{\bm{s}}$ by its homotopy type. Then the Manolescu-Ozsv\'ath surgery complex becomes a \emph{perturbed surgery foumula} \cite{Liu}. More concretely,  for a 2-component L--space link $\L$, we replace the complexes $\A^{\epsilon_{1}, \epsilon_{2}}_{\bm{s}}$ by
$$H_{\ast}(\A^{\epsilon_{1}\epsilon_{2}}_{\bm{s}})\cong \F[[U]]. $$
We replace the edge maps $\Phi^{\pm L_{i}}_{\bm{s}}$ as follows:
\begin{eqnarray}
\label{maps2}
\begin{aligned}
\Phi^{L_{1}}_{s_{1}, s_{2}}=U^{H(s_{1}, s_{2})-H(\infty, s_{2})}, \quad &\Phi^{-L_{1}}_{s_{1}, s_{2}}: U^{H(-s_{1}, -s_{2})-H(\infty, -s_{2})}, \\
\Phi^{L_{2}}_{s_{1}, s_{2}}: U^{H(s_{1}, s_{2})-H(s_{1}, \infty)}, \quad &\Phi^{-L_{2}}_{s_{1}, s_{2}}=U^{H(-s_{1}, -s_{2})-H(-s_{1}, \infty)}, \\
\Phi^{L_{1}}_{s_{1}}=U^{H(s_{1}, \infty)}=U^{H_{1}(s_{1}-l/2)}, \quad &\Phi^{-L_{1}}_{s_{1}}=U^{H_{1}(l/2-s_{1})}, \\
\Phi^{L_{2}}_{s_{2}}=U^{H( \infty, s_{2})}=U^{H_{2}(s_{2}-l/2)}, \quad &\Phi^{-L_{2}}_{s_{2}}=U^{H_{2}(l/2-s_{2})}, 
\end{aligned}
\end{eqnarray}
where $H(s_{1}, s_{2}), H_{1}(s_{1})$ and $H_{2}(s_{2})$ are H-functions of $\L, L_{1}$ and $L_{2}$, respectively. 

\begin{figure}[H]
\centering
\includegraphics[width=3.0in]{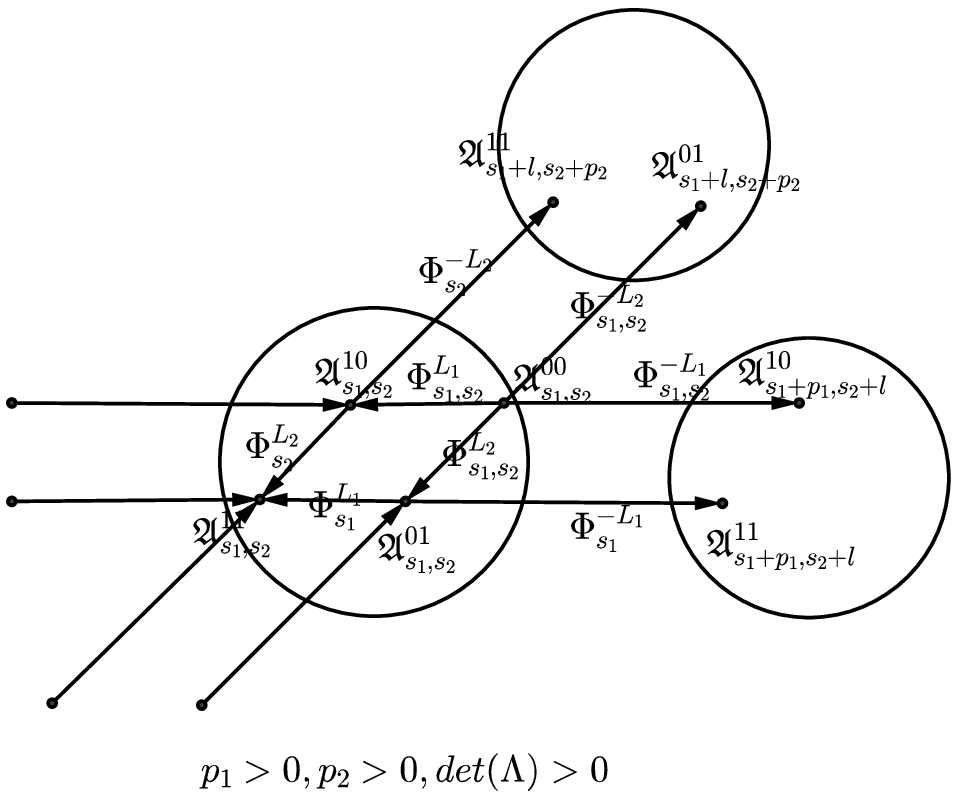}
\caption{ \label{surgery map}}

\end{figure}

We denote the perturbed complex as $\widetilde{C}(\Lambda)$, and it is chain homotopy equivalent to the original chain complex $\C(\bH^{\L}, \Lambda)$ as $\F[[U]]$-modules. Hence, $H_{\ast}(\widetilde{C}(\Lambda))\cong HF^{-}(S^{3}_{\Lambda}(\L))$ as an $\F[[U]]$-module \cite{Liu, MO}. 

The surgery complex splits as a direct sum corresponding to Spin$^{c}$-structures. Recall that for the surgery matrix $\Lambda$ associated to $\L$, there is an identification: Spin$^{c}(S^{3}_{\Lambda}(\L))=\H(\L)/ H(L ,\lambda)$, where $H(L, \lambda)$ is the lattice spanned by $\Lambda$ \cite{MO}.

Now we  review the truncated perturbed surgery complex. We refer the reader to \cite{Liu, MO} for details. The constant $b$ in Lemma \ref{isomorphism} determines a parallelogram $Q$ in the plane, with vertices $P_{1}, P_{2}, P_{3}, P_{4}$ counterclockwise  labelled, satisfying the following condition: The point $P_{i}$ has the coordinate $(x_{i}, y_{i})$ such that 
\begin{equation}
\label{condition2}
\begin{cases}
   x_{1}>b \\
   y_{1}>b,
\end{cases}
\begin{cases}
   x_{2}<-b \\
   y_{2}>b,
\end{cases}
\begin{cases}
   x_{3}<-b \\
   y_{3}<-b,
\end{cases}
\begin{cases}
   x_{4}>b \\
   y_{4}<-b.
\end{cases}
\end{equation}
We also require that every edge of $Q$ is either parallel to the vector $\Lambda_{1}$ with length greater than $|| \Lambda_{1}||$ or parallel to $\Lambda_{2}$ with length greater than $|| \Lambda_{2}||$. The way of doing truncation is not unique. We follow the way Y. Liu did in \cite{Liu}. One can choose the parallelogram $Q$ to be centered at the origins as follows. Let 
$$\{ P_{1}, P_{2}, P_{3}, P_{4} \}=\left\lbrace \dfrac{ i_{0}\Lambda_{1}+j_{0}\Lambda_{2}}{2}, \dfrac{ -i_{0}\Lambda_{1}+j_{0}\Lambda_{2}}{2}, \dfrac{ i_{0}\Lambda_{1}-j_{0}\Lambda_{2}}{2}, \dfrac{- i_{0}\Lambda_{1}-j_{0}\Lambda_{2}}{2} \right\rbrace, $$
with $i_0, j_0$ being positive integers, such that  \eqref{condition2} holds.

Instead of using the constant $b$ to truncate the surgery complex, one can also use different constants $b'_{1}, b'_{2}\in \mathbb{N}$ to truncate the complex in vertical and in horizontal directions. Then $\Phi^{\pm L_{i}}_{\psi^{\pm M}(\bm{s})}$ induces an isomorphism on homology whenever $|s_{i}|>b'_{i}$ and $L_{i}$ has the orientation corresponding to the sign of $s_{i}$ for $i=1, 2$. Let $b_{i}$ be the minimal number among the choices of $b'_{i}$. For 2-component L--space links $\L$, we can use \eqref{maps2} to define $b_{i}$ in terms of the $H$-function. 

\begin{definition}
\label{def:b}
For an oriented 2-component L--space link $\L$ with linking number $l$, we define:
$$b_{1}(\L)=\min \{ \lceil s_{1}-1\rceil \mid H(s_{1}, s_{2})=H(\infty, s_{2}) \textup{ for all } s_{2} \},$$
$$b_{2}(\L)=\min \{ \lceil s_{2}-1 \rceil \mid H(s_{1}, s_{2})=H(s_{1}, \infty) \textup{ for all } s_{1} \}.$$
\end{definition}

Fix the surgery matrix  $\Lambda$. Now we review the finitely generated surgery complex after truncation in the  Spin$^{c}$-structure  $\mathfrak{u}\in \H(\L)/ H(\L, \Lambda)$. For details, see \cite{Liu}. Let $S^{\epsilon_{1}\epsilon_{2}}$ denote the collection of summands $\A^{\epsilon_{1}\epsilon_{2}}_{\bm{s}}$ of the truncated surgery complex in the Spin$^{c}$-structure $\mathfrak{u}$ where $\epsilon_{1}, \epsilon_{2}=0$ or $1$.

Suppose 
$$\bm{s}=\theta_{1}\Lambda_{1}+\theta_{2}\Lambda_{2}\in \mathfrak{u}, \quad P_{1}=a_{1}\Lambda_{1}+a_{2}\Lambda_{2}. $$
Denote 
$$A_{1}=\lceil -\theta_{1}-|a_{1}|\rceil, \quad A_{2}=\lfloor -\theta_{1}+|a_{1}|\rfloor, $$
$$B_{1}=\lceil -\theta_{2}-|a_{2}| \rceil, \quad B_{2}=\lfloor -\theta_{2}+|a_{2}|\rfloor.$$
Based on the signs of $d_{1}, d_{2}$ and $\det\Lambda$, there are six cases for the truncated regions.

\textbf{Case 1:} $d_{1}>0, d_{2}>0, \det(\Lambda)>0$. 
$$S^{00}(\Lambda, \mathfrak{u})=\mathfrak{u} \cap Q, \quad S^{10}(\Lambda, \mathfrak{u})=\mathfrak{u}\cap Q\cap (Q+\Lambda_{1}), $$
$$S^{01}(\Lambda, \mathfrak{u})=\mathfrak{u}\cap Q\cap (Q+\Lambda_{2}), \quad S^{11}(\Lambda, \mathfrak{u})=\mathfrak{u}\cap Q\cap (Q+\Lambda_{1}+\Lambda_{2}).$$ 

\textbf{Case 2:} $d_{1}<0, d_{2}<0, \det(\Lambda)>0$. 
$$S^{00}(\Lambda, \mathfrak{u})=\mathfrak{u}\cap Q, \quad S^{01}(\Lambda, \mathfrak{u})=\mathfrak{u}\cap \{Q\cup (Q+\Lambda_{1})\},$$
$$S^{01}(\Lambda, \mathfrak{u})=\mathfrak{u}\cap \{Q\cup (Q+\Lambda_{2})\}, $$
$$\quad S^{11}(\Lambda, \mathfrak{u})=\mathfrak{u}\cap \{Q\cup (Q+\Lambda_{1})\cup (Q+\Lambda_{2})\cup (Q+\Lambda_{1}+\Lambda_{2})\}.$$

\textbf{Case 3:} $d_{1}>0, d_{2}<0$. 
$$S^{00}(\Lambda, \mathfrak{u})=\mathfrak{u}\cap Q, \quad S^{10}(\Lambda, \mathfrak{u})=\mathfrak{u}\cap \{ Q\cap (Q+\Lambda_{1}) \}, $$
$$S^{01}(\Lambda, \mathfrak{u})=\mathfrak{u}\cap \{ Q\cup (Q+\Lambda_{2})\},$$
$$ \quad S^{11}(\Lambda, \mathfrak{u})=\mathfrak{u}\cap \{[Q\cup (Q+\Lambda_{2})]\cap ([Q\cup (Q+\Lambda_{2})]+\Lambda_{1}) \}.$$

\textbf{Case 4:} $d_{1}<0, d_{2}>0$. 
$$S^{00}(\Lambda, \mathfrak{u})=\mathfrak{u}\cap Q, \quad S^{10}(\Lambda, \mathfrak{u})=\mathfrak{u}\cap \{Q\cup (Q+\Lambda_{1})\}, $$
$$S^{01}(\Lambda, \mathfrak{u})=\mathfrak{u}\cap \{Q\cap (Q+\Lambda_{2})\},$$
$$ \quad S^{11}(\Lambda, \mathfrak{u})=\mathfrak{u}\cap \{[Q\cap (Q+\Lambda_{2})]\cup ([Q\cap (Q+\Lambda_{2})]+\Lambda_{1})\}.$$

\textbf{Case 5:}  $l>0, \det(\Lambda)<0$. 
$$S^{00}(\Lambda, \mathfrak{u})=\mathfrak{u}\cap Q, \quad S^{10}(\Lambda, \mathfrak{u})=(\mathfrak{u}\cap Q\cap (Q+\Lambda_{1}))\cup T^{10},$$
$$S^{01}(\Lambda, \mathfrak{u})=(\mathfrak{u}\cap Q \cap (Q+\Lambda_{2}))\cup T^{01},$$
$$ \quad S^{11}(\Lambda, \mathfrak{u})=\mathfrak{u}\cap Q\cap (Q+\Lambda_{1}+\Lambda_{2}), $$
where $T^{10}=\{\bm{s}+A_{2}\Lambda_{1}+B_{1}\Lambda_{2}\}, T^{01}=\{\bm{s}+A_{1}\Lambda_{1}+B_{2}\Lambda_{2}\}$. 

\textbf{Case 6:} $l<0, \det(\Lambda)<0$. 
$$S^{00}(\Lambda, \mathfrak{u})=\mathfrak{u}\cap Q \cap (Q-\Lambda_{1}-\Lambda_{2}),$$
$$ \quad S^{10}(\Lambda, \mathfrak{u})=(\mathfrak{u}\cap Q\cap (Q-\Lambda_{1}))\cup T^{10}, $$
$$S^{01}(\Lambda, \mathfrak{u})=(\mathfrak{u}\cap Q \cap (Q-\Lambda_{2}))\cup T^{01}, \quad S^{11}(\Lambda, \mathfrak{u})=\mathfrak{u}\cap Q, $$
where $T^{10}=\bm{s}+A_{1}\Lambda_{1}+B_{2}\Lambda_{2}, T^{01}=\bm{s}+B_{1}\Lambda_{2}+A_{1}\Lambda_{1}$. 

Denote 
$$\bar{C}^{\epsilon_{1}\epsilon_{2}}(\Lambda, \mathfrak{u})=\bigoplus_{i\in \Z} \bigoplus_{ j\in \Z, \bm{s}+i\Lambda_{1}+j\Lambda_{2}\in S^{\epsilon_{1}\epsilon_{2}}} \A^{\epsilon_{1}\epsilon_{2}}_{\bm{s}+i\Lambda_{1}+j\Lambda_{2}}, $$
where $\epsilon_{1}, \epsilon_{2}\in \{0, 1\}$. 

The truncated complex is defined as 
$$\bar{C}(\bH, \Lambda, \mathfrak{u})=\bigoplus_{\epsilon_{1}, \epsilon_{2}\in \{0, 1\}} \bar{C}^{\epsilon_{1}\epsilon_{2}}(\Lambda, \mathfrak{u}).$$
The differential is obtained by restricting $D$ to $\bar{C}(\bH, \Lambda, \mathfrak{u})$. Up to homotopy equivalence, we simply regard $\A^{\delta_{1}\delta_{2}}_{\bm{s}+i\Lambda_{1}+j\Lambda_{2}}$ as its homotopy type $\F[[U]]$ and the differentials as the ones defined in \eqref{condition2}. It is homotopy equivalent to $(\C(\bH^{\L}, \Lambda, \mathfrak{u}), D)$. Hence the homology of the truncated perturbed complex is isomorphic to $HF^{-}(S^{3}_{\Lambda}(\L), \mathfrak{u})$ \cite{MO},  up to some grading shift. Since we are working on truncated surgery complexes from here on, it suffices to consider polynomials over $\F[U]$. 

By putting $U=0$, we get the chain complex of $\F$-vector spaces $\widehat{C}(\Lambda, \mathfrak{u})$ whose homology is isomorphic to $\widehat{HF}(S^{3}_{\Lambda}(\L), \mathfrak{u})$. Note that the differential $\Phi^{\pm L_{i}}_{\bm{s}}$ will be replaced by $\widehat{\Phi}^{\pm L_{i}}_{\bm{s}}$ which should be either 0 or 1 from $\F$ to $\F$.


\subsection{The associated CW-complexes}  
\label{location}
In this section, we associate a finite rectangular CW-complex to the truncated surgery complex. We refer the reader to \cite[Section 3.3]{GLM} for more details. Each $\A^{00}_{\bm{s}}$ in the truncated surgery complex corresponds to a 2-cell. Each $\A^{01}_{\bm{s}}$ and $\A^{10}_{\bm{s}}$ corresponds to a 1-cell, and each $\A^{11}_{\bm{s}}$ corresponds to a 0-cell with the boundary map specified by \eqref{maps2}. 

According to the different signs of $d_1, d_2$ and $\det \Lambda$, there are six cases for the truncation process described in Section \ref{perturbed}. In all these cases, the associated CW-complex is a rectangle $R$ on a square lattice, with some parts of the boundary erased. Consider the chain complex $C$ generated by the squares, edges, and vertices of $R$ over $\F$ with the usual differential $\partial $. Then the homology of $C$ is isomorphic to the homology of $R$ relative to the erased parts of the boundary. More precisely, we will have the following three situations:

\begin{itemize}
\item[(a)]  For case 1 in Section \ref{perturbed}, the CW-complex $R$ is a rectangle with all 1-cells and 0-cells  on the boundary  erased as shown  in Figure \ref{sides}. Then $(R, \partial R)\simeq (S^{2}, pt)$. Therefore $H_2(C,\partial)\cong\F$ is generated by the sum of all 2-cells, and all other homologies vanish. 
\item[(b)] For case 2 in Section \ref{perturbed}, the CW-complex $R$ is a rectangle with none of the cells erased  in Figure \ref{sides}. Then $R$ is contractable, so $H_{0}(C, \partial )\cong \F$ is generated by the class of a 0-cell, and all other homologies vanish. 
\item[(c)] For other 4 cases in Section \ref{perturbed}, the CW complex $R$ is a rectangle with some 1-cells and 0-cells erased on the boundary in Figure \ref{sides}. Then $R$ relative to the erased cells is homotopy equivalent to $(S^{1}, pt)$. Therefore, $H_{1}(C, \partial)\cong \F$ is generated by the class of any path connecting erased boundaries, and all other homologies vanish. 
\end{itemize}

\begin{figure}[H]
\centering
\includegraphics[width=1.1in]{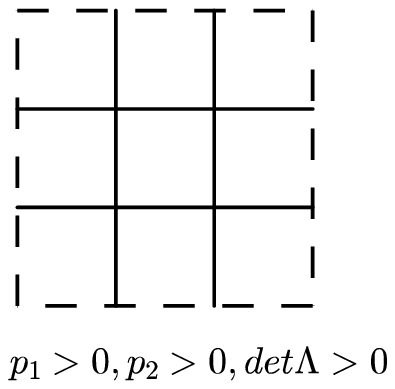}
\hspace{0.3in}
\includegraphics[width=1.1in]{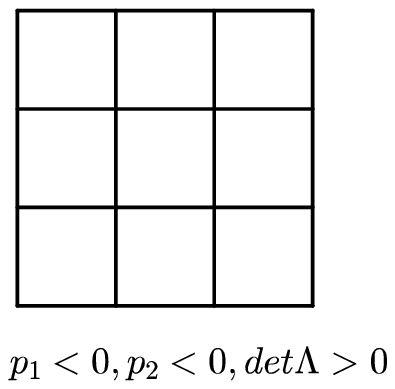} 
\hspace{0.3in}
\includegraphics[width=1.1in]{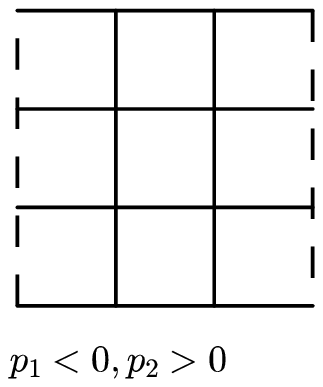}
\hspace{0.3in}
\includegraphics[width=1.0in]{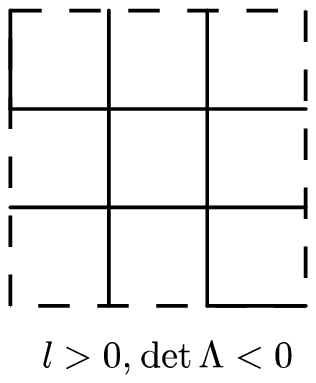}

\caption{ \label{sides}}
\end{figure}

 We associate a $(2-\epsilon_{1}-\epsilon_{2})$-dimensional cell in $C$ to $\A^{\epsilon_{1}\epsilon_{2}}_{\bm{s}}$ for $\epsilon_{1}, \epsilon_{2}\in \{0, 1\}$. One can construct a chain map from the truncated surgery complex $\bar{C}(\bH, \Lambda, \mathfrak{u})$ to the cell-complex $C$, see \cite[Section 4]{GLM}. Each cell $\square$ in $(C, \partial)$ corresponds to a copy of $\F[U]\cong H_{\ast}(\A^{\epsilon_{1}\epsilon_{2}}_{\bm{s}})$ generated by some element $z(\square)$ where the cell $\square$ is associated to $\A^{\epsilon_{1}\epsilon_{2}}_{\bm{s}}$. We denote the homological grading of $z(\square)$ by $\deg(\square)$. Recall that $U$ has homological grading $-2$. Then the degree of $z(\square)U^{k}\in H_{\ast}(\A^{\epsilon_{1}\epsilon_{2}}_{\bm{s}})$ equals $\deg(\square)-2k$, and we call $z(\square)U^{k}$ the graded lift of the cell $\square$ of degree $\deg(\square)-2k$. 

\begin{theorem}\cite[Corollary 4.3]{GLM}
\label{cellhom}
The free part of the homology $H_{\ast}(\bar{C}(\bH, \Lambda, \mathfrak{u}), D)/Tors$ is generated by the graded lifts of representatives of homology classes in $H_{\ast}(C, \partial)$. Two classes are equivalent if and only if they have the same degree and lift the same homology class. 
\end{theorem}

Recall that $\bar{C}(\bH, \Lambda, \mathfrak{u})=\bigoplus_{\epsilon_{1}, \epsilon_{2}}\bar{C}^{\epsilon_{1}\epsilon_{2}}(\Lambda, \mathfrak{u})$. Up to homotopy equivalence, we can regard 
$$\bar{C}^{\epsilon_{1}\epsilon_{2}}(\Lambda, \mathfrak{u})=\bigoplus_{i\in \Z}\bigoplus_{j\in \Z, \bm{s}+i\Lambda_{1}+j\Lambda_{2}\in S^{\epsilon_{1}\epsilon_{2}}} H_{\ast}(\A^{\epsilon_{1}\epsilon_{2}}_{\bm{s}+i\Lambda_{1}+j\Lambda_{2}}).$$

\begin{corollary}
\label{generation1}
\begin{itemize}
\item[(a)] If $d_{1}>0, d_{2}>0, \det(\Lambda)>0$, then the free part of $HF^{-}(S^{3}_{d_{1}, d_{2}}(\L), \mathfrak{u})$ is generated by a chain in $\bar{C}^{00}(\Lambda, \mathfrak{u})$. 
\item[(b)] If $d_{1}<0, d_{2}<0, \det(\Lambda)>0$, then the free part of  $HF^{-}(S^{3}_{d_{1}, d_{2}}(\L), \mathfrak{u})$ is generated by a chain in $\bar{C}^{11}(\Lambda, \mathfrak{u})$. 
\item[(c)] For the rest of the cases,  the free part of  $HF^{-}(S^{3}_{d_{1}, d_{2}}(\L), \mathfrak{u})$ is generated by a chain in $\bar{C}^{01}(\Lambda, \mathfrak{u})\oplus \bar{C}^{10}(\Lambda, \mathfrak{u})$.
\end{itemize}
\end{corollary}

\begin{proof}
The proof is straight-forward by using  the surgery theorem  $HF^{-}(S^{3}_{d_{1}, d_{2}}(\L), \mathfrak{u})\cong H_{\ast}(\bar{C}(\bH, \Lambda, \mathfrak{u}))$ \cite{MO} and Theorem \ref{cellhom}.

\end{proof}

 Recall that  $\widehat{HF}(S^{3}_{d_{1}, d_{2}}(\L), \mathfrak{u})$ is isomorphic to the homology of the chain complex of $\F$-vector spaces $\widehat{C}(\Lambda, \mathfrak{u})$ which is obtained from $\bar{C}(\bH, \Lambda, \mathfrak{u})$ by putting $U=0$ \cite{Liu}. Suppose that $H_{\ast}(\A^{\epsilon_{1}\epsilon_{2}}_{\bm{s}})\cong \F[U]$ is generated by $z^{\epsilon_{1}\epsilon_{2}}_{\bm{s}}$ as a $\F[U]$-module. We let $\widehat{\A}^{\epsilon_{1}\epsilon_{2}}_{\bm{s}}\cong \F$ denote the vector space generated by $z^{\epsilon_{1}\epsilon_{2}}_{\bm{s}}$. Then 
$\widehat{C}( \Lambda, \mathfrak{u})=\bigoplus_{\epsilon_{1}, \epsilon_{2}}\widehat{C}^{\epsilon_{1}\epsilon_{2}}(\Lambda, \mathfrak{u})$
where 
$$\widehat{C}^{\epsilon_{1}\epsilon_{2}}(\Lambda, \mathfrak{u})=\bigoplus_{i\in \Z}\bigoplus_{j\in \Z, \bm{s}+i\Lambda_{1}+j\Lambda_{2}\in S^{\epsilon_{1}\epsilon_{2}}} \widehat{\A}^{\epsilon_{1}\epsilon_{2}}_{\bm{s}+i\Lambda_{1}+j\Lambda_{2}}.$$
 
\begin{corollary}
\label{generation2}
For a 2-component L--space link $\L=L_{1}\cup L_{2}$ with linking number $l$, suppose that $S^{3}_{d_{1}, d_{2}}(\L)$ is an L--space. For any Spin$^{c}$-structure $\mathfrak{u}$, we have:

\begin{itemize}
\item[(a)] If $d_{1}>0, d_{2}>0, \det(\Lambda)>0$, then $\widehat{HF}(S^{3}_{d_{1}, d_{2}}(\L), \mathfrak{u})\cong \F$ is generated by a chain in $\widehat{C}^{00}(\Lambda, \mathfrak{u})$. 
\item[(b)] If $d_{1}<0, d_{2}<0, \det(\Lambda)>0$, then $\widehat{HF}(S^{3}_{d_{1}, d_{2}}(\L), \mathfrak{u})$ is generated by a chain in $\widehat{C}^{11}(\Lambda, \mathfrak{u})$. 
\item[(c)] For the rest of the cases,    $\widehat{HF}(S^{3}_{d_{1}, d_{2}}(\L), \mathfrak{u})$ is generated by a chain in $\widehat{C}^{01}(\Lambda, \mathfrak{u})\oplus \widehat{C}^{10}(\Lambda, \mathfrak{u})$.

\end{itemize}
\end{corollary}

\begin{proof}
If $S^{3}_{d_{1}, d_{2}}(\L)$ is an L--space,  $HF^{-}(S^{3}_{d_{1}, d_{2}}(\L), \mathfrak{u})\cong \F[U]$. Note that $\widehat{HF}(S^{3}_{d_{1}, d_{2}}(\L), \mathfrak{u})$ is obtained from $HF^{-}(S^{3}_{d_{1}, d_{2}}(\L), \mathfrak{u})\cong \F[U]$ by putting $U=0$. By Corollary \ref{generation1}, the tower $\F[U]$ is generated by a chain in $\bar{C}^{00}(\Lambda, \mathfrak{u})$, in $\bar{C}^{11}(\Lambda, \mathfrak{u})$ or in $\bar{C}^{01}(\Lambda, \mathfrak{u})\oplus \bar{C}^{01}(\Lambda, \mathfrak{u})$. Without loss of generality, we assume the chain is in $\bar{C}^{00}(\Lambda, \mathfrak{u})$, and it is written as:
$$z^{00}_{\bm{s}_{1}}+z^{00}_{\bm{s}_{2}}+\cdots+ U \bar{z}^{00}_{\mathfrak{u}}$$
where $\bar{z}^{00}_{\mathfrak{u}}$ is a chain in $\bar{C}^{00}(\Lambda, \mathfrak{u})$. By putting $U=0$, the generator becomes $z^{00}_{\bm{s}_{1}}+z^{00}_{\bm{s}_{2}}+\cdots$ which is a chain in $\widehat{C}^{00}(\Lambda, \mathfrak{u})$, and it generates $\widehat{HF}(S^{3}_{d_{1}, d_{2}}(\L), \mathfrak{u})$. 

\end{proof}

\section{L--space surgeries on 2-component L--space links}
In this section, we start with type (A) 2-component L--space links, and the proof of Theorem \ref{ppositive}. 

\subsection{Type (A) 2-component L--space links}
\label{typea}
Recall that a 2-component L--space link $\L=L_{1}\cup L_{2}$ is type (A) if there exists a lattice point $\bm{s}=(s_{1}, s_{2})\in \H(\L)$  such that $H_{\L}(\bm{s})>H_{\L}(s_{1}, s_{2}+1), H_{\L}(\bm{s})> H_{\L}(s_{1}+1, s_{2})$ and one of  $H_{\L}(\infty, s_{2}), H_{\L}(s_{1}, \infty)$ equals $0$. Otherwise, it is called a type (B) link. 
This lattice point is  a very good point in the convention of \cite{GN}, and they proved Theorem \ref{ppositive} for very good points. We give another method to prove it here. 

\medskip

\textbf{Proof of Theorem \ref{ppositive}:} By Theorem \ref{L-space link pro}, $L_{1}$ and $L_{2}$ are both L--space knots.  Then $H_{L_{i}}(s_{i})=0$ if and only if  $s_{i}\geq g_{i}$ for $i=1, 2$ where $g_{i}$ is the genus of $K_{i}$ by Lemma \ref{genus0}. Without loss  of generality, we assume $H_{\L}(s_{1}, \infty)=0$. By Lemma \ref{infinity},  $H_{\L}(s_{1}, \infty)=H_{L_{1}}(s_{1}-l/2)=0$. Then $s_{1}-l/2\geq g_{1}$ by Lemma \ref{genus0}, which indicates that $s_{1}\geq l/2+g_{1}$. 

Recall that $\Phi^{L_{1}}_{\bm{s}}=U^{H(s_{1}, s_{2})-H(\infty, s_{2})}$. By Lemma \ref{growth control}, $H(s_{1}, s_{2})>H(s_{1}+1, s_{2})\geq H(\infty, s_{2})$. So 
$$H(s_{1}, s_{2})-H(\infty, s_{2})>0.$$
 Let $U=0$. We have $\widehat{\Phi}^{L_{1}}_{\bm{s}}=0$.  Similarly, we can prove $\widehat{\Phi}^{L_{2}}_{\bm{s}}=0$. The map $\Phi^{-L_{1}}_{\bm{s}}$ equals $U^{H(-s_{1}, -s_{2})-H(\infty, -s_{2})}$. By Lemma \ref{symmetry3}, 
$$H(-s_{1}, -s_{2})-H(\infty, -s_{2})=H(s_{1}, s_{2})-H(\infty, s_{2}+l)+s_{1}-l/2.$$
In this paper, we orient the link $\L$ so that the linking number $l$ is nonnegative. Then $H(s_{1}, s_{2})>H(\infty, s_{2})\geq H(\infty, s_{2}+l)$ by Lemma \ref{growth control}. Combining with $s_{1}\geq l/2+g_{1}$, we have 
$$H(-s_{1}, -s_{2})-H(\infty, -s_{2})>0. $$ 
 Hence, $\widehat{\Phi}^{-L_{1}}_{\bm{s}}=0$. For the map $\Phi^{-L_{2}}_{\bm{s}}=U^{H(-s_{1}, -s_{2})-H(-s_{1}, \infty)}$, we have
$$H(-s_{1}, -s_{2})-H(-s_{1}, \infty)=H(s_{1}, s_{2})-H(s_{1}+l, \infty)+s_{2}-l/2.$$
Similarly, we have $H(s_{1}+l, \infty)\leq H( s_{1}, \infty)=0$ and 
$$H(s_{1}, s_{2})+s_{2}-l/2\geq H( \infty, s_{2})+1+s_{2}-l/2=H_{2}(s_{2}-l/2)+1+s_{2}-l/2=H_{2}(l/2-s_{2})+1>0.$$
Hence, $\widehat{\Phi}^{-L_{2}}_{\bm{s}}=0$. Then $\widehat{\A}^{00}_{\bm{s}}$ is the generator of $\widehat{HF}(S^{3}_{d_{1}, d_{2}}(\L), \mathfrak{s})$. By Corollary \ref{generation2}, this is possible only when $d_{1}>0, d_{2}>0$ and $\det \Lambda>0$.   \qed

\medskip

\begin{lemma}
\label{boundb}
For a 2-component L--space link $\L=L_{1}\cup L_{2}$ with linking number $l$, $b_{i}(\L)\geq g_{i}-1+l/2$ where $g_{i}$ is the genus of $L_{i}$ and $i=1, 2$. 
\end{lemma}

\begin{proof}
By definition \ref{def:b}, $b_{1}=\min\{ \lceil s_{1}-1 \rceil \mid H(s_{1}, s_{2})=H(\infty, s_{2}) \textup{ for all } s_{2} \}$. Then $H(s_{1}, \infty)=H(\infty, \infty)=0$ for all $s_{1}>b_{1}$. By Lemma \ref{infinity}, $H(s_{1}, \infty)=H_{L_{1}}(s_{1}-l/2)=0$ for all $s_{1}>b_{1}$. Then Lemma \ref{genus0} implies that $s_{1}-l/2\geq g_{1}$. If the linking number $l$ is even, then $s_{1}\in \Z$. By replacing $s_{1}$ by $b_{1}+1$, we have
$$b_{1}\geq g_{1}-1+l/2 .$$
If $l$ is odd, we let $s_{1}=b_{1}+1/2$. Then 
$$b_{1}\geq g_{1}-1/2+l/2.$$
In both cases, we have $b_{1}\geq g_{1}-1+l/2$. The argument for $b_{2}$ is similar.  

\end{proof}

\begin{corollary}
\label{positive surgery 2}
Let $\L=L_{1}\cup L_{2}$ be an L--space link. Suppose $b_{i}\geq g_{i}+l/2$ for $i=1$ or $2$.  If $S^{3}_{d_{1}, d_{2}}(\L)$ is an L--space, then $d_{1}>0, d_{2}>0$ and $\det \Lambda>0$. 
\end{corollary}

\begin{proof}
Without loss of generality, we assume that $b_{1}\geq g_{1}+l/2$. Suppose that the linking number $l$ is even. Then $(b_{1}, \infty)\in \H(\L)$, and $H(b_{1}, \infty)=H_{1}(b_{1}-l/2)=0$.  By the definition of $b_{1}$ ( see Definition \ref{def:b}), there exists $y$ such that
$$H(b_{1}, s_{2})=H(b_{1}+1, s_{2}) , \textup{ and } H(b_{1}, y)\neq H(b_{1}+1, y),$$
for all $s_{2}>y$.  Assume that $H(b_{1}, y+1)=H(b_{1}+1, y+1)=a$ for some $a\geq 0$. Then $H(b_{1}+1, y)=a$ or $a+1$. If it equals $a+1$, then $H(b_{1}, y)=a+2$ which contradicts to Lemma \ref{growth control}. Hence,  
$$H(b_{1}, y)=a+1, H(b_{1}, y+1)=H(b_{1}+1, y)=H(b_{1}+1, y+1)= a.$$
Then  the link $\L$ is type (A). By Theorem \ref{ppositive}, if $S^{3}_{d_{1}, d_{2}}(\L)$ is an L--space, $d_{1}>0, d_{2}>0$  and $\det \Lambda>0$. 

If $l$ is odd, then $(b_{1}-1/2, \infty)\in \H(\L)$ and $b_{1}\geq g_{1}+l/2+1/2$. So $H(b_{1}-1/2, \infty)=H_{1}(b_{1}-1/2-l/2)=0$. The rest of the argument is the same as the one in  the case that $l$ is even. 
\end{proof}

\begin{proposition}
\label{char of typeA}
A 2-component L--space link $\L$ is type (A) if and only if $b_{i}\geq g_{i}+l/2$ for $i=1$ or $2$.
\end{proposition}

\begin{proof}
 The ``if " part can be seen from the proof of Corollary \ref{positive surgery 2}. For the ``only if " part, we assume that the link is type (A). Then  there exists a lattice point $(s_{1}, s_{2})\in \H(\L)$ such that $H(s_{1}, \infty)=0$, $H(s_{1}, s_{2})>H(s_{1}+1, s_{2})$ and $H(s_{1}, s_{2})>H(s_{1}, s_{2}+1)$. So $s_{1}\leq b_{1}$. Note that $H(s_{1}, \infty)=H_{1}(s_{1}-l/2)=0$. By Lemma \ref{genus0}, we have $s_{1}-l/2\geq g_{1}$. This implies $b_{1}\geq g_{1}+l/2$.

\end{proof}

\begin{definition}
For a 2-component L--space link $\L=L_{1}\cup L_{2}$,  a lattice point $(s_{1}, s_{2})\in \H(\L)$ is called  a maximal lattice point if $H(s_{1}, s_{2})=1, H(s_{1}+1, s_{2})=H(s_{1}, s_{2}+1)=0$. 
\end{definition}

We refer the readers to \cite{Liu2} for more details about maximal lattice points. If there exists a maximal lattice point $(s_{1}, s_{2})$ for $\L$, then $b_{i}\geq s_{i}$ for $i=1, 2$. Note that $H(s_{1}, \infty)\leq H(s_{1}, s_{2}+1)=0$. Then $H(s_{1}, \infty)=H_{1}(s_{1}-l/2)=0$, which indicates that $s_{1}\geq g_{1}+l/2$. By Proposition \ref{char of typeA},  $\L$ is type (A).

\begin{lemma}
\label{knot L--space}
Assume that $\L=L_{1}\cup L_{2}$ is a type (A) L--space link. If $S^{3}_{d_{1}, d_{2}}(\L)$ is an L--space, then either $S^{3}_{d_{1}}(L_{1})$ or $S^{3}_{d_{2}}(L_{2})$ is an L--space. 

\end{lemma}

\begin{proof}
By Theorem \ref{ppositive}, if $S^{3}_{d_{1}, d_{2}}(\L)$ is an L--space, then $d_{1}>0, d_{2}>0$ and $\det\Lambda>0$. If $L_{1}$ or $L_{2}$ is an unknot, the statement is clear. Now we assume that both $L_{1}$ and $L_{2}$ are not unknots with genera $g_{1}, g_{2}\geq 1$. Suppose that $S^{3}_{d_{1}}(L_{1})$ and $S^{3}_{d_{2}}(L_{2})$ are not L--spaces. Then $d_{1}\leq 2g_{1}-2$ and $d_{2}\leq 2g_{2}-2$. Pick the lattice point $\bm{s}=(g_{1}-1+l/2, g_{2}-1+l/2)\in \H(\L)$. Note that 
$$\Phi^{L_{1}}_{s_{1}}=U^{H(g_{1}-1+l/2, \infty)}=U^{H_{1}(g_{1}-1)}, \quad \Phi^{L_{2}}_{s_{2}}=U^{H_{2}(g_{2}-1)},$$
$$\Phi^{-L_{1}}_{s_{1}-d_{1}}=U^{H_{1}(d_{1}-g_{1}+1)}, \quad \Phi^{-L_{2}}_{s_{2}-d_{2}}=U^{H_{2}(d_{2}-g_{2}+1)}.$$

Since $d_{i}\leq 2g_{i}-2$ for $i=1, 2$, $d_{i}-g_{i}+1\leq g_{i}-1$. Hence, 
$$H_{i}(g_{i}-1)>0, \quad H_{i}(d_{i}-g_{i}+1)>0,$$
for $i=1, 2$. Therefore, $\widehat{\Phi}^{L_{i}}_{s_{i}}=\widehat{\Phi}^{-L_{i}}_{s_{i}-d_{i}}=0$ for $i=1, 2$. This indicates that $\widehat{\A}^{11}_{s_{1}, s_{2}}$ is the generator  of $HF^{-}(S^{3}_{d_{1}, d_{2}}(\L), \bm{s})$ which is a chain in $\widehat{C}^{11}(\Lambda, \bm{s})$ (see Figure \ref{surgery map}). However, we have $d_{1}>0, d_{2}>0$ and $\det \Lambda>0$. By Corollary \ref{generation2},  $\widehat{HF}(S^{3}_{d_{1}, d_{2}}(\L), \bm{s})$ is generated by a chain in $\widehat{C}^{00}(\Lambda, \bm{s})$. So we get a contradiction, and either $S^{3}_{d_{1}}(L_{1})$ or $S^{3}_{d_{2}}(L_{2})$ is an L--space. 

\end{proof}

\subsection{Type (B) 2-component L--space links}
\label{typeb}
In this section, $\L=L_{1}\cup L_{2}$ denotes a type (B) L--space link. We start with large  surgeries (i.e $d_{i}\gg 0$ or $ d_{i}\ll0$) on $\L$.

\begin{lemma}
\label{2unknots}
Let $\L=L_{1}\cup L_{2}$ be a type (B) L--space link. If $S^{3}_{d_{1}, d_{2}}(\L)$ is an L--space with $d_{1}< -2b_{1}-l,  d_{2}<-2b_{2}-l$, then both $L_{1}$ and $L_{2}$ are unknots. 

\end{lemma}

\begin{figure}[H]
\centering
\includegraphics[width=2.0in]{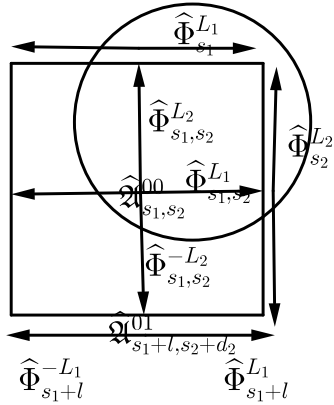}
\caption{Truncated surgery complex  \label{surgery}}
\end{figure} 

\begin{proof}
Pick the lattice point $\bm{s}=(-l/2, l/2+1)\in \H(\L)$. Since $d_{1}<-2b_{1}-l, d_{2}<-2b_{2}-l$,  we can choose the parallelogram $Q$  in the plane such that the point $\bm{s}$ is in $Q$, and other lattice points in the same Spin$^{c}$ structure are outside of $Q$. Then after truncation, in this  particular Spin$^{c}$ structure, we have the  parallelogram in Figure \ref{surgery} for the truncated surgery complex.

Note that 
$$\Phi^{-L_{2}}_{\bm{s}}=U^{H(\bm{-s})-H(-s_{1}, \infty)}=U^{H(\bm{s})-H(s_{1}+l, \infty)+s_{2}-l/2}.$$
Here $H(s_{1}, s_{2})\geq H(s_{1}+l, s_{2})\geq H(s_{1}+l, \infty)$ by Lemma \ref{growth control} and $s_{2}-l/2=1$. Then $\widehat{\Phi}^{-L_{2}}_{\bm{s}}=0$. We also have 
$$\Phi^{L_{1}}_{s_{1}+l}=U^{H_{1}(s_{1}+l/2)}=U^{H_{1}(0)}, \quad  \Phi^{-L_{1}}_{s_{1}+l}=U^{H(0)}.$$
If $L_{1}$ is not an unknot, then $H_{1}(0)\neq 0$. Let $U=0$, we have $\widehat{\Phi}^{\pm L_{1}}_{s_{1}+l}=0$. Hence $\widehat{\A}^{01}_{\bm{s}+\Lambda_{2}}$ generates $\widehat{HF}(S^{3}_{d_{1}, d_{2}}(\L), \bm{s})$ which is a chain in $\widehat{C}^{01}(\Lambda, \bm{s})$.  However, note  that $d_{1}<0, d_{2}<0, \det \Lambda>0$. By Corollary \ref{generation2},   $\widehat{HF}(S^{3}_{d_{1}, d_{2}}(\L), \bm{s})$ is generated by a chain in $\widehat{C}^{11}(\Lambda, \bm{s})$, which is a contradiction.  Hence $L_{1}$ is an unknot. Similarly, we can  prove that $L_{2}$ is also an unknot.

\end{proof}

\begin{proposition}
\label{unknot component2}
Let $\L=L_{1}\cup L_{2}$ be a type (B) L--space link. If $S^{3}_{d_{1}, d_{2}}(\L)$ is an L--space with $d_{1}>2b_{1}+l, d_{2}<-2b_{2}-l$, then $L_{2}$ is an unknot.  

\end{proposition}

\begin{proof}
Suppose that $S^{3}_{d_{1}, d_{2}}(\L)$ is an L--space for $d_{1}>2b_{1}+l$ and $d_{2}<-2b_{2}-l$ and $L_{2}$ is not an unknot. Consider the lattice point $(s_{1}, s_{2})=(g_{1}+l/2, l/2)$. Since $L_{2}$ is not an unknot,  $H_{2}(0)\geq 1$.  Then $H(g_{1}+l/2, l/2)\geq H(\infty, l/2)=H_{2}(0)\geq 1$.  Since  $d_{1}>2b_{1}+l$ and $d_{2}<-2b_{2}-l$, we can truncate the surgery complex to be Figure \ref{truncatedpn} in the spin$^{c}$-structure represented by  $(g_{1}+l/2, l/2)$. 
Then $\widehat{\Phi}^{\pm L_{1}}_{g_{1}+l/2, l/2}=0$ since the the pair of vertical sides in the truncated surgery complex are erased. Observe that 
$$\Phi^{L_{2}}_{g_{1}+l/2, l/2}=U^{H(g_{1}+l/2, l/2)-H(g_{1}+l/2, \infty)}, \quad \Phi^{-L_{2}}_{g_{1}+l/2, l/2}=U^{H(-g_{1}-l/2, -l/2)-H(-g_{1}-l/2, \infty)}.$$
By Corollary \ref{symmetry3}, $H(-g_{1}-l/2, -l/2)-H(-g_{1}-l/2, \infty)=H(g_{1}+l/2, l/2)-H(g_{1}+3l/2, \infty)$. 
Note that $H(g_{1}+3l/2, \infty)\leq H(g_{1}+l/2, \infty)=H_{1}(g_{1})=0$. Hence, $\widehat{\Phi}^{\pm L_{i}}_{g_{1}+l/2, l/2}=0$, and $\widehat{\A}^{00}_{g_{1}+l/2, l/2}$ generates $\widehat{HF}(S^{3}_{d_{1}, d_{2}}(\L), (g_{1}+l/2, l/2))$ as in Figure \ref{truncatedpn}.  By Corollary \ref{generation2}, this is possible only when $d_{1}>0, d_{2}>0$ and $\det \Lambda>0$. So we get a contradiction, and  $L_{2}$ is an unknot. 

\end{proof}

\begin{remark}
The similar result holds for the component $L_{1}$. 
\end{remark}

\begin{figure}[H]
\centering
\includegraphics[width=2.0in]{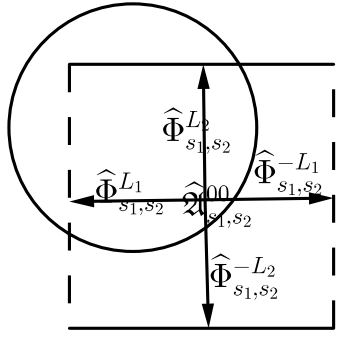}
\caption{Truncated surgery complex  \label{truncatedpn}}
\end{figure}

\begin{corollary}
\label{cor:unknot}
Let $\L=L_{1}\cup L_{2}$ be an L--space link. If $S^{3}_{d_{1}, d_{2}}(\L)$ is an L--space with $d_{1}>2b_{1}+l,  d_{2}<0$, then $L_{2}$ is an unknot. 
\end{corollary}

\begin{proof}
By Lemma \ref{boundb}, $b_{1}\geq g_{1}-1+l/2$. Then $d_{1}>2b_{1}\geq 2g_{1}-2$. So $d_{1}\geq 2g_{1}-1$. Hence $S^{3}_{d_{1}}(L_{1})$ is an L--space. By the surgery induction (Lemma \ref{induction}), $S^{3}_{d_{1}, d_{2}-k}(\L)$ is an L--space for any $k>0$. Note that $d_{2}-k<-2b_{2}-l$ for sufficiently large $k$. So $S^{3}_{d_{1}, d'_{2}}(\L)$ is an L--space for $d'_{2}<-2b_{2}-l$ and $d_{1}>2b_{1}+l$. By Proposition \ref{unknot component2}, $L_{2}$ is an unknot.

\end{proof}

\begin{figure}[H]
\centering
\includegraphics[width=2.5in]{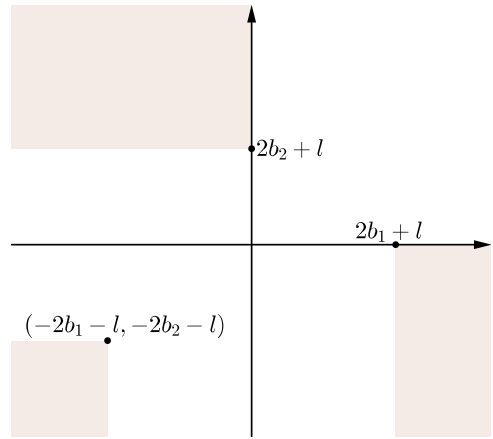}
\caption{L--space surgeries \label{F7}}
\end{figure}

\begin{corollary}
Suppose $\L=L_{1}\cup L_{2}$ is an L--space link such that both components are not unknots. Then the possible L--space surgeries are indicated by the white regions in Figure \ref{F7}. 

\end{corollary}

\begin{proof}
It is straight-forward from Lemma \ref{2unknots} and Corollary \ref{cor:unknot}.

\end{proof}

\begin{example}
The torus link $\L=T(4, 6)$ is 2-component L--space link with linking number $6$. Both of the knot components are right-handed trefoils. Its L--space surgery set is contained in the white region indicated in Figure \ref{F7}, and is unbounded from below. We refer the readers to \cite[Figure 1]{GN} for the details.

 \end{example}

Now, we characterize the torus link $T(2, 2l)$. The torus link $T(2, 2l)$ is a 2-component L--space link with unknotted components and linking number $l$ \cite{Liu}. The Hopf link admits negative L--space surgeries (i.e $d_{1}\ll 0, d_{2}\ll 0$) which are connect sums of lens spaces, but the torus links $T(2, 2l)$ with $l>1$ do not admit such L--space surgeries (Proposition \ref{linking number 1}). Moreover,  we prove that large negative surgeries (i.e $d_{i}\ll 0$)  characterize the Hopf link (Corollary \ref{hopf}), and the existence of L--space surgeries $S^{3}_{d_{1}, d_{2}}(\L)$ with $d_{1}d_{2}<0$  characterizes the torus link $T(2, 2l)$ with $l>1$ (Theorem \ref{torus}). 

\begin{figure}[H]
\centering
\includegraphics[width=2.0in]{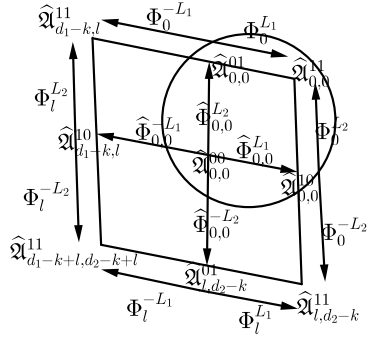}
\caption{Truncated  surgery complex  \label{surgery22}}
\end{figure}

\begin{proposition}
\label{linking number 1}
Let $\L$ be a nontrivial 2-component L--space link with unknotted components and $l>1$. Then $S^{3}_{d_{1}, d_{2}}(\L)$ is not an L--space for $d_{1}<0, d_{2}<0$ and $\det \Lambda>0$. 

\end{proposition}

\begin{proof}
We first assume that the linking number is even. Then pick the lattice point $(0, 0)\in \H(\L)$. If $S^{3}_{d_{1}, d_{2}}(\L)$ is an L--space for $d_{1}<0, d_{2}<0, \det\Lambda>0$, by the surgery induction (Lemma \ref{induction}), $S^{3}_{d_{1}-k, d_{2}-k}(\L)$ is also an L--space for any integer $k>0$. By choosing $k$ sufficiently large, we can truncate the surgery complex in the Spin$^{c}$ structure of $(0, 0)$ as in Figure \ref{surgery22}. 

Observe that $\Phi^{\pm L_{i}}_{0, 0}=U^{H(0, 0)-H(0, \infty)}$ and $H(0, \infty)=H_{1}(-l/2)=l/2$. We claim that   $H(0, 0)=H(0, \infty)=H(\infty, 0)=l/2$. If $H(0, 0)>l/2$, then $H(0, 0)-H(0, \infty)>0$ and $H(0, 0)-H(\infty, 0)>0$. Hence,  $\widehat{\Phi}_{0, 0}^{\pm L_{i}}=0$. This indicates that $\widehat{\A}^{00}_{0, 0}$ generates $\widehat{HF}(S^{3}_{d_{1}-k, d_{2}-k}(\L), \bm{0})$. By Corollary \ref{generation2}, this is possible only when $d_{1}>0, d_{2}>0$ and $\det\Lambda>0$. Hence, $H(0, 0)=l/2$, and $\widehat{\Phi}_{0, 0}^{\pm L_{i}}=1$. Note that 
$$\Phi^{\pm L_{i}}_{0}=U^{H_{i}(\mp l/2)}, \Phi^{\pm L_{i}}_{l}=U^{H_{i}(\pm l/2)}$$
for $i=1, 2$. 

Recall that  $H_{i}(l/2)=0$ and $H_{i}(-l/2)> 0$. Then $\widehat{\Phi}^{L_{i}}_{0}=0$,  $\widehat{\Phi}^{-L_{i}}_{0}=1$,  $\widehat{\Phi}_{l}^{L_{i}}=1$ and $ \widehat{\Phi}_{l}^{-L_{i}}=0$.
 Hence, we see that $\widehat{\A}^{11}_{0, 0}$ and $\widehat{\A}^{11}_{d_{1}+l-k, d_{2}+l-k}$ generate $\widehat{HF}(S^{3}_{d_{1}-k, d_{2}-k}(\L), \bm{0})$ in Figure \ref{surgery22},  which contradicts to the assumption that $S^{3}_{d_{1}-k, d_{2}-k}(\L)$ is an L--space. 

Now we suppose that the linking number $l$ is odd, i.e. $l\geq 3$. The argument is very similar. Pick the lattice point $(1/2, 1/2)\in \H(\L)$. Similarly, we have
$$\Phi^{ L_{i}}_{1/2}=U^{H_{1}(1/2-l/2)}, \quad \Phi^{-L_{i}}_{1/2}=U^{H_{1}(l/2-1/2)}, $$
$$\Phi^{L_{i}}_{1/2+l}=U^{H_{1}(1/2+l/2)}, \quad \Phi^{-L_{i}}_{1/2+l}=U^{H_{1}(-1/2-l/2)}.$$
Hence $\widehat{\Phi}^{L_{i}}_{1/2}=\widehat{\Phi}^{-L_{i}}_{1/2+l}=0$. Therefore, $\widehat{\A}^{11}_{1/2, 1/2}$ and $\widehat{\A}^{11}_{1/2+d_{1}+l-k, 1/2+d_{2}+l-k}$ generate the homology $\widehat{HF}(S^{3}_{d_{1}-k, d_{2}-k}(\L), (1/2, 1/2))$ which is a contradiction. Hence $S^{3}_{d_{1}, d_{2}}(\L)$ is not an L--space for $d_{1}<0, d_{2}<0, \det\Lambda>0$. 

\end{proof}

\begin{figure}[H]
\centering
\includegraphics[width=3.0in]{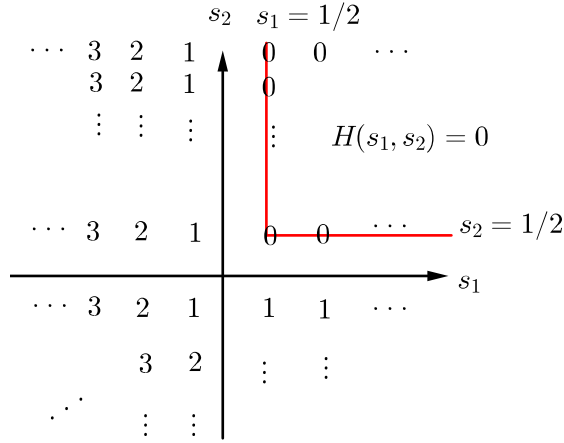}
\caption{The $H$-function $H(\L)$  \label{h-function}}
\end{figure}

\begin{corollary}
\label{hopf}
Suppose that $\L=L_{1}\cup L_{2}$ is a nontrivial 2-component L--space link. If $S^{3}_{d_{1}, d_{2}}(\L)$ is an L--space with $d_{1}<-2b_{1}-l, d_{2}<-2b_{2}-l$, then $\L$ is the Hopf link. 
\end{corollary}

\begin{proof}
By Lemma \ref{2unknots}, both $L_{1}$ and $L_{2}$ are unknots. There is an L--space surgery with $d_{1}<0, d_{2}<0$ and $\det \Lambda>0$. By Proposition \ref{linking number 1}, the linking number $l$ is $0$ or $1$. If the linking number is 0, $S^{3}_{d_{1}, d_{2}}(\L)$ is an L--space if and only if  $d_{1}>2b_{1}$ and $d_{2}>2b_{2}$ \cite[Theorem 5.1]{GLM}. Then the linking number must be $1$.  Note that $H(s_{1}, \infty)=H_{1}(s_{1}-1/2)=0$ and $H(\infty, s_{2})=H_{2}(s_{2}-1/2)=0$ for all $s_{1}>0, s_{2}>0$. The link $\L$ has no maximal lattice point since it is type (B). Then $H(s_{1}, s_{2})=0$ for all $(s_{1}, s_{2})\in \H(\L)$ and $(s_{1}, s_{2})\succeq \bm{0}$. 

By Lemma  \ref{symmetry2}, $H(-s_{1}, -s_{2})=H(s_{1}, s_{2})+s_{1}+s_{2}$ for all $(s_{1}, s_{2})\in \H(\L)$ and $(s_{1}, s_{2})\succeq \bm{0}$. It is not hard to see the H-function of $\L$ is as shown in Figure \ref{h-function} by Lemma \ref{growth control}. The Alexander polynomial for the Hopf link is $1$. By \eqref{Alex} and \eqref{computeh}, one can compute the $H$-function of the Hopf link is the same as the one in Figure \ref{h-function}. This indicates that they have the same Thurston polytope \cite{Liu1}, and the Thurston polytope lies on a line of slope 1 passing through the origin. Then there exists an annulus  representing the homology class $(-1, 1)\in H_{2}(S^{3}, \L; \Z)$ whose boundary components are longitudes for the corresponding link components. Hence $\L$ is the Hopf link. 

\end{proof}

\begin{theorem}
\label{torus}
Let $L=L_{1}\cup L_{2}$ be an L--space link with unknotted components and linking number $l$. If $S^{3}_{d_{1}, d_{2}}(\L)$ is an L--space for $d_{1}d_{2}<0$, then $L$ is the torus link  $T(2, 2l)$. 

\end{theorem}

\begin{proof}
Without loss of generality, assume that $S^{3}_{d_{1}, d_{2}}(\L)$ is an L--space for $d_{1}>0$ and $d_{2}<0$.  Then $\L$ is type (B) and there are no maximal lattice points. Note that both $L_{1}$ and $L_{2}$ are unknots, so $S^{3}_{d_{i}}(L_{i})$ is also an L--space. By the surgery induction (Lemma \ref{induction}), $S^{3}_{d'_{1}, d'_{2}}(\L)$ is also an L--space for $d'_{1}\gg 0$ and $d'_{2}\ll 0$.

We first suppose that the linking number $l>0$ is even.  In the Spin$^{c}$ structure $(0, 0)$, we can truncate the surgery complex to be a  square with a pair of sides erased as in Figure \ref{truncatedpn}. Then $\widehat{\Phi}^{\pm L_{1}}_{0, 0}=0$. Note that 
$$\Phi^{\pm L_{2}}_{0, 0}=U^{H(0, 0)-H(0, \infty)}.$$
Since $L_{1}$ is unknot, then $H( 0, \infty)=H_{1}(-l/2)=l/2$. We claim that $H(0, 0)=l/2$. Otherwise, $H(0, 0)-H(0, \infty)>0$ and $\widehat{\Phi}^{\pm L_{2}}_{0, 0}=0$. So $\widehat{\A}^{00}_{0, 0}$ generates $\widehat{HF}(S^{3}_{d'_{1}, d'_{2}}(\L), (0, 0))$. By Corollary \ref{generation2}, this is only possible when $d'_{1}>0, d'_{2}>0$,  which is a contradiction. Hence $H(\infty, 0)=H(0, 0)=H(0, \infty)=l/2$. 

We claim the $H$-function of $\L$ is the same as the $H$-function of the torus link $T(2, 2l)$. By Lemma \ref{growth control}, $H(0, s_{2})=H(s_{1}, 0)=l/2$ for all $s_{1}, s_{2}\geq 0$. Note that $H(-s_{1}, \infty)=H_{1}(-s_{1}-l/2)=s_{1}+l/2$ for all $s_{1}>0$. Then $H(s_{1}, s_{2})=H(s_{1}, \infty)$ for all $s_{1}\leq 0, s_{2}\geq 0$ by Lemma \ref{growth control}. By a similar argument, we can prove that  $H(s_{1}, s_{2})=H(\infty, s_{2})$ for all $s_{1}\geq 0, s_{2}\leq 0$. 

Now we analyze the $H$-function in the first quadrant (i.e $s_{1}>0, s_{2}>0$). Observe that $H(s_{1}, \infty)=0$ for all $s_{1}\geq l/2$ and $H(s_{1}, \infty)=l/2-s_{1}$  for $0\leq s_{1}<l/2$ since $H(s_{1}, \infty)=H_{1}(s_{1}-l/2)$ and $L_{1}$ is an unknot. The similar result holds for $H(\infty, s_{2})$. Since $\L$ is type (B), $H(s_{1}, s_{2})=0$ for all $s_{1}\geq l/2,  s_{2}\geq l/2$. By Lemma \ref{growth control}, we have 
$$H(s_{1}, s_{2})=H(s_{1}, \infty) \textup{ for } 0\leq s_{1}\leq l/2,  s_{2}\geq l/2,$$
$$H(s_{1}, s_{2})=H( \infty, s_{2}) \textup{ for } 0\leq s_{2}\leq l/2,  s_{1}\geq l/2.$$
Now we just need to discuss the values of the $H$-function in the square region $\{(s_{1}, s_{2})| 0\leq s_{1}\leq l/2, 0\leq s_{2}\leq l/2\}.$

The values of the $H$-function at the boundary are already known. We claim that the diagonal value $H(k, k)=l/2-k$. For the Spin$^{c}$ structure $(k, k)$, we can truncate the surgery complex in this Spin$^{c}$ structure to be the square as shown in Figure \ref{truncatedpn}. Then 
$$\Phi^{L_{2}}_{k, k}=U^{H(k, k)-H(k, \infty)}, \quad \Phi^{-L_{2}}_{k, k}=U^{H(-k, -k)-H(-k, \infty)}=U^{H(k, k)-H( k+l, \infty)+k-l/2}. $$
If $S^{3}_{d'_{1}, d'_{2}}(\L)$ is an L--space, $\widehat{\Phi}^{L_{2}}_{k, k}=1$ or $\widehat{\Phi}^{-L_{2}}_{k, k}=1$ by a similar argument. This indicates that 
$$H(k, k)=H(k, \infty)=H_{1}(k-l/2)=l/2-k$$
or
$$ H(k, k)=H( k+l, \infty)-k+l/2=H_{1}(k+l/2)-k+l/2=l/2-k.$$
Thus in both cases $H(k, k)=l/2-k$. Note that $H(k, l/2)=H(l/2, k)=l/2-k$. Hence, 
$$H(s_{1}, k)=H(k, s_{2})=l/2-k$$
for all $k\leq s_{1}\leq l/2, k\leq s_{2}\leq l/2$. 

By Corollary \ref{symmetry3}, the values of $H$-function in the third quadrant are determined by its values in the first quadrant. 
Therefore, the $H$-function of $\L$ is the same as the one of the torus link $T(2, 2l)$. This means that they have the same Thurston polytope \cite{Liu1}, and the Thurston polytope lies on a line of slope 1 passing the origin. Then there exists an annulus  representing the homology class $(-1, 1)$ whose boundary components are longitudes for the corresponding link components. Hence $\L$ is the torus link $T(2, 2l)$. 

\end{proof}

\subsection{L--space surgeries on L--space links with vanishing linking numbers}
\label{sec:linkingnumber0}

In this section, we discuss 2-component L--space links $\L=L_{1}\cup L_{2}$ with vanishing linking number. With this additional assumption, some results in Section \ref{typea} and Section \ref{typeb} can be strengthened.

Lemma \ref{knot L--space} is strengthened as follows:

\begin{proposition}\cite[Proposition 5.6]{GLM}
\label{knot link}
Suppose that $\L$ is an L--space link with linking number zero. If $S^{3}_{d_{1}, d_{2}}(\L)$ is an L--space, then either $S^{3}_{d_{1}}(L_{1})$ or $S^{3}_{d_{2}}(L_{2})$ is an L--space. 
\end{proposition}

We strengthen  Corollary \ref{cor:unknot}, proving that $\L$ is the disjoint union of $L_{1}$ and an unknot:

\begin{proposition}
\label{unknot component}
Suppose the link $\L=L_{1}\cup L_{2}$ is an $L$-space link  with vanishing linking number.  If $S^{3}_{d_{1}, d_{2}}(\L)$ is an $L$-space for $d_{1}>2b_{1}$ and $d_{2}<0$, then $\L=L_{1} \sqcup U$ where $U$ is an unknot.
\end{proposition}

\begin{proof}
By Corollary \ref{cor:unknot}, $L_{2}$ is an unknot. Next, we prove that $H(s_{1}, s_{2})=H_{1}(s_{1})+H_{2}(s_{2})$ where $H_{1}, H_{2}$ are the $H$-functions of $L_{1}, L_{2}$ respectively. By a similar argument to the one in Corollary \ref{cor:unknot}, we assume that $S^{3}_{d_{1}, d_{2}}(\L)$ is an L--space with $d_{1}>2b_{1}, d_{2}<-2b_{2}$. Then  in each Spin$^{c}$ structure $(s_{1}, s_{2})$, we can truncate the surgery complex to be the square  with a pair of sides erased as shown  in Figure  \ref{truncatedpn}. So $\widehat{\Phi}^{\pm L_{1}}_{s_{1}, s_{2}}=0$.

Suppose that $s_{2}>0$. Then 
$$H(-s_{1}, -s_{2})-H(-s_{1}, \infty)=H(s_{1}, s_{2})-H_{1}(s_{1})+s_{2}>0.$$
Hence $\P_{s_{1}, s_{s}}^{-L_{2}}=0$. We claim that $\P^{L_{2}}_{s_{1}, s_{2}}=1$. Otherwise, $\widehat{\A}^{00}_{s_{1}, s_{2}}$ generates $\widehat{HF}(S^{3}_{d_{1}, d_{2}}(\L), \bm{s})$ which is a contradiction by a similar argument to the one in Proposition \ref{unknot component2}. Recall that $\Phi_{s_{1}, s_{2}}^{L_{2}}=U^{H(s_{1}, s_{2})-H(s_{1}, \infty)}$. So $\widehat{\Phi}^{L_{2}}_{s_{1}, s_{2}}=1$ indicates that:
$$H(s_{1}, s_{2})=H(s_{1}, \infty)=H_{1}(s_{1}),$$
for all $s_{2}>0$. 

Suppose that $s_{2}=0$. Then 
$$H(s_{1}, s_{2})-H_{1}(s_{1})=H(-s_{1}, -s_{2})-H_{1}(-s_{1}). $$
This indicates that $\widehat{\Phi}^{L_{2}}_{s_{1}, 0}=\widehat{\Phi}^{-L_{2}}_{s_{1}, 0}$. By a similar argument, we prove that they equal $1$.  
Hence, $H(s_{1}, 0)=H(s_{1}, \infty)=H_{1}(s_{1})$.

Now we consider the case that $s_{2}<0$. By Lemma \ref{growth control}, 
$H(-s_{1}, -s_{2})-H(-s_{1}, \infty)\geq 0$. Then 
$$H(-s_{1}, -s_{2})-H_{1}(-s_{1})=H(s_{1}, s_{2})-H_{1}(s_{1})+s_{2}\geq 0.$$
So $H(s_{1}, s_{2})-H_{1}(s_{1})\neq 0$, and $\widehat{\Phi}^{L_{2}}_{s_{1}, s_{2}}=0$. By a similar argument, one has $\widehat{\Phi}^{-L_{2}}_{s_{1}, s_{2}}=1$. This indicates that 
$$H(-s_{1}, -s_{2})-H(-s_{1}, \infty)=1. $$
Hence $H(s_{1}, s_{2})-H_{1}(s_{1})=-s_{2}$. Recall that $L_{1}$ is an unknot. So $H_{2}(s_{2})=(|s_{2}|-s_{2})/2$, and
 $H(s_{1}, s_{2})=H_{1}(s_{1})+H_{2}(s_{2})$ . 

By \eqref{computeh}, the Alexander polynomial $\Delta_{\L}(t_{1}, t_{2})$ vanishes.  Then the Thurston polytope of $\L$ is the same as that of $L_{1}\sqcup U$, which is an interval on the $s_{1}$-axis connecting $(-g(L_{1}), 0)$ and $(g(L_{1}), 0)$ \cite{Liu1}. The Thurston norm in $(0, 1)$ direction is $0$, and in $(1, 0)$ direction is $g(L_{1})$. It is not hard to use the definition of Thurston norm and the computation of Euler characteristics of surfaces to prove:  $L_{1}$ and $L_{2}$ bound pairwise disjoint surfaces with genera $g(L_{1})$  and $0$ in $S^{3}$, respectively. Hence,  $\L$ is the disjoint union of $L_{1}$ and $U$.

\end{proof} 

Next, we discuss positive L--space surgeries on 2-component L--space links with linking number zero. 

\begin{lemma}
\label{gap distance}
Assume that $\L=L_{1}\cup L_{2}$ is a nontrivial L--space link with vanishing linking number. If $S^{3}_{d_{1}, d_{2}}(\L)$ is an L--space for $d_{1}>0, d_{2}>2b_{2}$, then $d_{1}>b_{1}$. 

\end{lemma}

\begin{proof}
By definition \ref{def:b}, $b_{1}=\min\{ \lceil s_{1}-1 \rceil \mid H(s_{1}, s_{2})=H(\infty, s_{2}) \textup{ for all } s_{2} \}$. Since the linking number is $0$, $\H(\L)\cong \Z^{2}$. Then there exists a lattice point $(b_{1}, s_{2})\in \H(\L)$ such that $H(b_{1}, s_{2})>H(\infty, s_{2})=H_{2}(s_{2})$. Since $d_{2}>2b_{2}$, we can truncate the surgery complex in the Spin$^{c}$ structure $(s_{1}, s_{2})$ as shown in Figure \ref{surgery3}. Observe that $\widehat{\Phi}^{-L_{1}}_{b_{1}, s_{2}}=0$, and $\widehat{\Phi}^{L_{1}}_{b_{1}, s_{2}}=0$ since $H(b_{1}, s_{2})>H_{2}(s_{2})$. Suppose that $s_{1}=b_{1}-d_{1}\geq 0$. Then $H(-s_{1}, -s_{2})-H_{2}(-s_{2})=H(s_{1}, s_{2})-H_{2}(s_{2})+s_{1}>0$ since $H(s_{1}, s_{2})\geq H(b_{1}, s_{2})>H_{2}(s_{2})$ by Lemma \ref{growth control}. This indicate that $\widehat{\Phi}^{-L_{1}}_{s_{1}, s_{2}}=0$. From Figure \ref{surgery3}, we see that $\widehat{\Phi}^{\pm L_{2}}_{s_{2}}=0$. So $\widehat{\A}^{00}_{b_{1}, s_{2}}$ and $\widehat{\A}^{10}_{b_{1}, s_{2}}$ are both generators of $\widehat{HF}(S^{3}_{d_{1}, d_{2}}(\L), (b_{1}, s_{2}))$, which contradicts to our assumption that $S^{3}_{d_{1}, d_{2}}(\L)$ is an L--space.  Hence, $b_{1}-d_{1}<0$. 

\begin{figure}[H]
\centering
\includegraphics[width=2.5in]{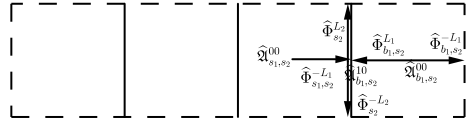}
\caption{Truncated  surgery complex  \label{surgery3}}
\end{figure}  

\end{proof}

\begin{lemma}
\label{one unknot}
Suppose that $\L=L_{1}\cup L_{2}$ is a nontrivial L--space link with vanishing linking number and $L_{2}$ is an unknot. If $S^{3}_{d_{1}, d_{2}}(\L)$ is an L--space for $d_{1}>0, d_{2}>b_{2}$, then $d_{1}>2b_{1}$. 

\end{lemma}

\begin{proof}
Suppose that $d_{1}\leq 2b_{1}$.  Pick the lattice point $(b_{1}, 0)$. Since $d_{2}>b_{2}$, we can truncate the surgery complex in the Spin$^{c}$ structure $(b_{1}, 0)$ to the rectangle with the boundary erased as in Figure \ref{surgery3}. Then $\widehat{\Phi}^{\pm L_{2}}_{b_{1}, 0}=0$.  We claim that   $H(b_{1}, 0)\neq 0$. Otherwise,  by Lemma \ref{growth control}, $H(b_{1}, 0)=0$ implies $H(b_{1}, s_{2})=H_{2}(s_{2})$ for any $s_{2}$. This contradicts to the definition of $b_{1}$. Then  $\Phi^{L_{1}}_{(b_{1}, 0)}=U^{H(b_{1}, 0)-H_{2}(0)}$ implies that $\widehat{\Phi}^{L_{1}}_{b_{1}, 0}=0$. Observe that 
$$\Phi^{-L_{1}}_{b_{1}-d_{1}, 0}=U^{H(d_{1}-b_{1}, 0)-H_{2}(0)}=U^{H(d_{1}-b_{1}, 0)}.$$
If $d_{1}\leq 2b_{1}$, then $d_{1}-b_{1}\leq b_{1}$. So $H(d_{1}-b_{1}, 0)\geq H(b_{1}, 0)>0$. Hence  $\widehat{\Phi}^{-L_{1}}_{b_{1}-d_{1}, 0}=0$. Combining with $\widehat{\Phi}^{L_{1}}_{b_{1}, 0}=0$, we prove that $\widehat{\A}^{10}_{b_{1}, 0}$ is the generator of  $\widehat{HF}(S^{3}_{d_{1}, d_{2}}(\L), (b_{1}, 0))$, which contradicts to Corollary \ref{generation2} and the assumption that $d_{1}>0, d_{2}>0, \det\Lambda>0$. 

\end{proof}

\begin{lemma}
\label{gap distance2}
Let $\L=L_{1}\cup L_{2}$ be a nontrivial L--space link with vanishing linking number. If $S^{3}_{d_{1}, d_{2}}(\L)$ is an L--space for $d_{1}>0, d_{2}>2b_{2}$, then $d_{1}\geq 2g_{1}-1$. 

\end{lemma}

\begin{proof}
If $L_{1}$ is an unknot, this is straight-forward. Suppose that $L_{1}$ is not an unknot and $d_{1}\leq 2g_{1}-2$. Pick the lattice point $(g_{1}-1, s_{2})$ such that $H_{2}(s_{2})=0$ and $H(g_{1}-1, s_{2})> 0$. This is possible since $H(g_{1}-1, s_{2})\geq H(g_{1}-1, \infty)>0$.  Since $d_{2}>2b_{2}$, we can truncate the surgery complex in each Spin$^{c}$ structure as shown in   Figure \ref{surgery3}. Observe that 
$$g_{1}-1-d_{1}\geq -g_{1}+1.$$
Let $s_{1}=g_{1}-1-d_{1}$. Then $s_{1}\geq -g_{1}+1$, and  
$$H(-s_{1}, -s_{2})-H_{2}(-s_{2})=H(s_{1}, s_{2})-H_{2}(s_{2})+s_{1}\geq H(s_{1}, \infty) +s_{1}=H_{1}(-s_{1})>0. $$
Note that $H(g_{1}-1, s_{2})-H_{2}(s_{2})>0$. Hence $\widehat{\Phi}^{L_{1}}_{g_{1}-1, s_{2}}=\widehat{\Phi}^{-L_{1}}_{g_{1}-1-d_{1}, s_{2}}=0$. This indicates that $\widehat{\A}^{10}_{g_{1}-1, s_{2}}$ is the generator of $\widehat{HF}(S^{3}_{d_{1}, d_{2}}(\L), (g_{1}-1, s_{2}))$, which is a contradiction by a similar argument as before.  
Therefore, $d_{1}\geq 2g_{1}-1$. 
\end{proof}

For a 2-component L--space link with vanishing linking number, $b_{i}\geq g_{i}-1$ for $i=1$ and $2$ by Lemma \ref{boundb}. We discuss L--space surgeries for such links based on the comparison of $b_{i}$ and $g_{i}-1$.

\begin{theorem}
\label{zero linking surgery2}
Let $\L=L_{1}\cup L_{2}$ be a 2-component L--space link with vanishing linking number. Suppose that $b_{i}=g_{i}-1$ for $i=1$ and $2$. Then $S^{3}_{d_{1}, d_{2}}(\L)$ is an L--space if and only if $d_{1}>2b_{1}$ and $d_{2}>2b_{2}$. 

\end{theorem}

\begin{figure}[H]
\centering
\includegraphics[width=2.0in]{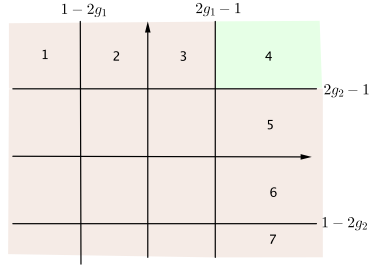} 

\caption{$b_{1}=g_{1}-1, b_{2}=g_{2}-1$ \label{F8}}
\end{figure}

\begin{proof} 
Since $b_{i}=g_{i}-1$,   $L_{i}$ is not an unknot for $i=1, 2$ and $2g_{i}-1=2b_{i}+1$. By Proposition \ref{knot link}, if $S^{3}_{d_{1}, d_{2}}(\L)$ is an L--space, then $(d_{1}, d_{2})$ must be in one of regions 1, 2, 3, 4, 5, 6, 7 in Figure \ref{F8}. For points $(d_{1}, d_{2})$ in region 4, $S^{3}_{d_{1}, d_{2}}(\L)$ is an L--space by large surgery formula. We use the green color for the region 4. By  Proposition \ref{unknot component}, points in regions 1, 2, 6,  7 won't give L--space surgeries since both of the components $L_{1}$ and $L_{2}$ are not unknots.  By Lemma \ref{gap distance2}, points in regions 3 and 5 don't give L--space surgeries. Hence $S^{3}_{d_{1}, d_{2}}(\L)$ is an L--space if and only if $d_{1}\geq 2g_{1}-1$ and $d_{2}\geq 2g_{2}-1$ in this case. 

\end{proof}

\begin{figure}[H]
\centering
\includegraphics[width=6.0in]{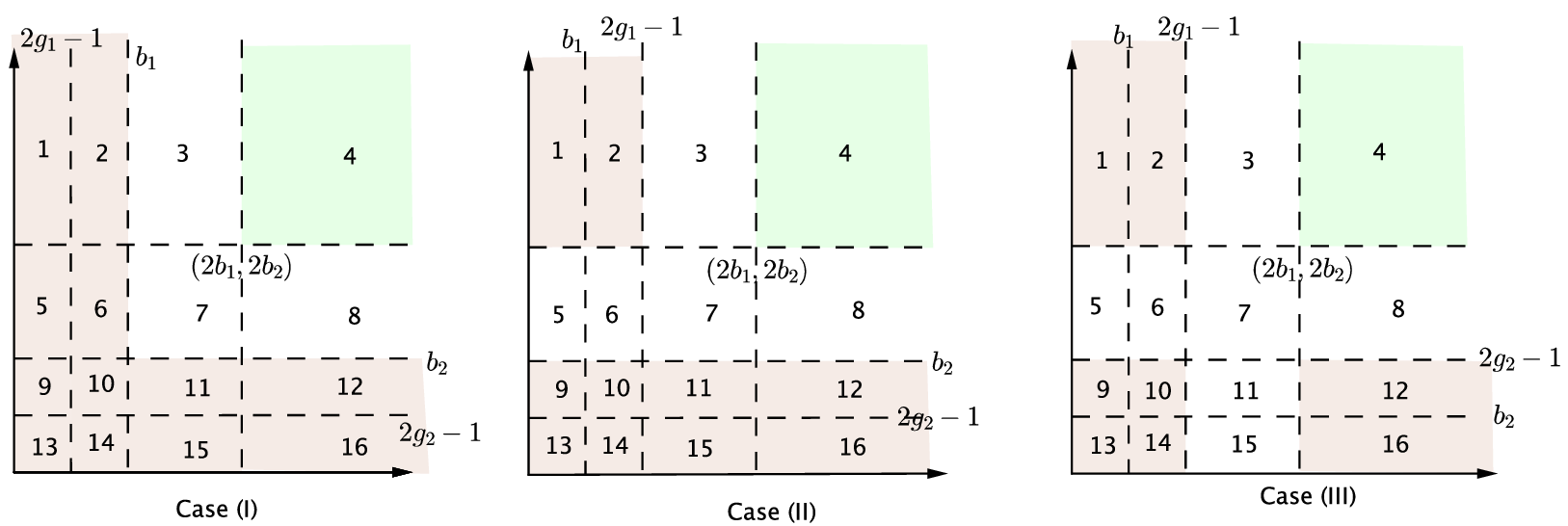}
\caption{L-space surgeries for $\L$ \label{F9}}
\end{figure}

\begin{theorem}
\label{zero linking surgeries}
Let $\L$ be a 2-component L--space link with vanishing linking number. Suppose that $b_{i}\geq g_{i}$ for $i=1$ or $2$.  The possible L--space surgeries are indicated by the white and green colored regions  in  Figure \ref{F9}. 

\end{theorem}

\begin{proof}
If $b_{i}\geq  g_{i}$ for $i=1$ or $2$, the link is type (A) by Proposition \ref{char of typeA}.  If $S^{3}_{d_{1}, d_{2}}(\L)$ is an L--space,  then  $d_{1}>0, d_{2}>0$. Based on the comparisons of $b_{i}$ and $2g_{i}-1$, we separate the discussion into three cases as shown in Figure \ref{F9}. In each figure, we use lines to separate the first quadrant into 16 regions. In Case (I), we suppose that $b_{i}>2g_{i}-1$ for $i=1$ and $2$ as shown in Figure \ref{F9}. If $(d_{1}, d_{2})$ is in region 4, then $S^{3}_{d_{1}, d_{2}}(\L)$ is an L--space by large surgery formula. So we use green color for this region. By Lemma \ref{gap distance} and Lemma \ref{gap distance2}, if $d_{i}>2b_{i}$, then $d_{i+1}> \max \{ b_{i+1}, 2g_{i+1}-2\}. $ This indicated that if $(d_{1}, d_{2})$ is in one of the regions 1, 2, 12, 16, then $S^{3}_{d_{1}, d_{2}}(\L)$ is not an L--space. We use red color to denote these regions.  If $(d_{1}, d_{2})$ is in region 6, 10 or 11, then by surgery induction (Lemma \ref{induction}), $S^{3}_{d_{1}+k_{1}, d_{2}+k_{2}}(\L)$ is also an L--space for any $k_{1}>0, k_{2}>0$, which contradicts to the fact that points in regions 2 and 12 don't give L--space surgeries . Hence no points in these three regions produce L--space surgeries. By Proposition \ref{knot link}, points in region 13 also cannot give L--space surgeries. If $(d_{1}, d_{2})$ is in region 5 or region 9, then by the surgery induction  again, $S^{3}_{d_{1}+k, d_{2}}$ is an L--space for any $k>0$. However, we know that points in regions 6 and 10 cannot produce L--space surgeries which is a contradiction. So we also use red color  for these regions to indicate the surgeries corresponding to them are not L--spaces. The similar argument works for the regions 14 and 15. Hence, we get Figure \ref{F9} Case (I). 

In Figure \ref{F9} Case (II), we suppose that $b_{1}<2g_{1}-1$ and $b_{2}>2g_{2}-1$. The argument for regions 1,2, 4, 12, 16, 13, 14, 9, 10, 11, 15 is very similar to the argument in Case (I). For points in regions 5 or 6, we cannot use the surgery  induction argument. For these points $(d_{1}, d_{2})$, it is possible that $S^{3}_{d_{1}, d_{2}}(\L)$ is  an L--space. In Figure \ref{F9} Case (III), we suppose that $b_{i}\leq  2g_{i}-1$ for both $i=1, 2$. We  prove that points in regions shaded by the red color won't give L--space surgeries by the similar argument to the one in Case (I). For points $(d_{1}, d_{2})$ in regions 5, 6, 11 and 15, it is possible that $S^{3}_{d_{1}, d_{2}}(\L)$ is an L--space. 
\end{proof}

\begin{corollary}
Assume that $\L=L_{1}\cup L_{2}$ is an L--space link with vanishing linking number and $L_{2}$ is unknot. If $S^{3}_{d_{1}, d_{2}}(\L)$ is an L--space, then $(d_{1}, d_{2})$ is in region 4 and 8 in Figure \ref{F9}.  
\end{corollary}

\begin{proof}
If $L_{2}$ is an unknot, we only need to consider Case (I) and Case (II) in Figure \ref{F9}. It suffices to consider the points in regions 5 and 6. By Lemma \ref{one unknot}, if $d_{2}>b_{2}$, then $d_{1}>2b_{1}$. Hence, $S^{3}_{d_{1}, d_{2}}(\L)$  cannot be an L--space for points in regions 5 and 6. 
\end{proof}

\medskip

\subsection{2-component L--space links with explicit descriptions of L--space surgeries}
\label{expdes}
In this section, we describe the L--space surgery set for some 2-component L--space links explicitly. They all have maximal lattice points, hence are type (A) L--space links.

For a 2-component L--space link $\L$, if there exists a maximal lattice point $\bm{s}=(s_{1}, s_{2})\in \H(\L)$, by the definition,  
$$H(s_{1}, s_{2})=1, \quad H(s_{1}+1, s_{2})=H(s_{1}, s_{2}+1)=H(s_{1}+1, s_{2}+1)=0.$$
Then $\chi(HFL^{-}(s_{1}+1, s_{2}+1))=-1$, and it is the coefficient of the term $t_{1}^{s_{1}+1/2}t_{2}^{s_{2}+1/2}$ in the symmetrized Alexander polynomial $\Delta_{\L}(t_{1}, t_{2})$ by \eqref{Alex} and \eqref{computeh}.  

\begin{proposition}
\label{symmetry}
Suppose that $\L=L_{1}\cup L_{2}$ is an L--space link with a maximal lattice point $(s_{1}, s_{2})\in \H(\L)$, and the coefficient of $t_{1}^{-s_{1}-1/2}t_{2}^{s_{2}+1/2}$ in the symmetrized Alexander polynomial $\Delta_{\L}(t_{1}, t_{2})$ is also nonzero. If $S^{3}_{d_{1}, d_{2}}(\L)$ is an L--space with $d_{2}>2b_{2}$, then $d_{1}>2s_{1}$. 

\end{proposition}

\begin{proof}
By Theorem \ref{ppositive}, if $S^{3}_{d_{1}, d_{2}}(\L)$ is an L--space, then $d_{1}>0, d_{2}>0$ and $\det \Lambda>0$. Since the coefficient of $t_{1}^{-s_{1}-1/2}t_{2}^{s_{2}+1/2}$ is nonzero, $\chi(HFL^{-}(-s_{1}, s_{2}+1))=\pm 1$. We first assume that $\chi(HFL^{-}(-s_{1}, s_{2}+1))= 1$. Then 
$$H(-s_{1}-1, s_{2})=H(-s_{1}-1, s_{2}+1)=H(-s_{1}, s_{2})=b+1, \quad H(-s_{1}, s_{2}+1)=b$$
for some $b\geq 0$. It is not hard to see that  $b_{2}\geq g_{2}+l/2$ by a similar argument to the one in Proposition \ref{char of typeA}. So $d_{2}\geq 2g_{2}-1$. If $S^{3}_{d_{1}, d_{2}}(\L)$ is an L--space, by  Lemma \ref{induction}, $S^{3}_{d_{1}+k, d_{2}}(\L)$ is also an L--space for all $k\geq 0$.  If $d_{1}\leq 2s_{1}$, then $S^{3}_{2s_{1}, d_{2}}(\L)$ is also an L--space. We can truncate the surgery complex to the rectangle with sides erased as in Figure \ref{surgery3} in the Spin$^{c}$ structure $(s_{1}, s_{2})$. Note that 
$$\Phi^{L_{1}}_{s_{1}, s_{2}}=U^{H(s_{1}, s_{2})-H(\infty, s_{2})}, \quad \Phi^{-L_{1}}_{-s_{1}, s_{2}}=U^{H(s_{1}, -s_{2})-H(\infty, -s_{2})}=U^{H(-s_{1}, s_{2})-H(\infty, s_{2}+l)-s_{1}+l/2}.$$
Here  $H(\infty, s_{2}+l)\leq H(\infty, s_{2})\leq H(s_{1}+1, s_{2})=0$. So  
\begin{eqnarray*}
	H(s_{1}, -s_{2})-H(\infty, -s_{2}) &=& H(-s_{1}, s_{2})-s_{1}+l/2\\
	&\geq & H(-s_{1}, \infty)+1-s_{1}+l/2 \\
	&=& H_{1}(-s_{1}-l/2)-s_{1}+l/2+1\\
	&=& H_{1}(s_{1}+l/2)+l+1>0.
\end{eqnarray*}
Hence, $\widehat{\Phi}^{L_{1}}_{s_{1}, s_{2}}=\widehat{\Phi}^{-L_{1}}_{-s_{1}, s_{2}}=0$. This indicates that $\widehat{\A}^{10}_{s_{1}, s_{2}}$ generates  $\widehat{HF}(S^{3}_{d_{1}, d_{2}}(\L), (s_{1}, s_{2}))$,  contradicting to Corollary \ref{generation2}. Thus, we have $d_{1}>2s_{1}$. 

Next we consider the case that $\chi(HFL^{-}(-s_{1}, s_{2}+1))=-1$. Then 
$$H(-s_{1}-1, s_{2})=b+1, \quad H(-s_{1}, s_{2})=H(-s_{1}, s_{2}+1)=H(-s_{1}-1, s_{2}+1)=b,$$
for some $b\geq 0$. By Lemma \ref{induction}, $S^{3}_{2s_{1}+1, d_{2}}(\L)$ is an L--space.  By a similar argument, we have
$$\widehat{\Phi}^{L_{1}}_{s_{1}, s_{2}}=\widehat{\Phi}^{-L_{1}}_{-s_{1}-1, s_{2}}=0,$$
which  contradicts to Corollary \ref{generation2} by a similar argument. Hence, $d_{1}>2s_{1}$. 
\end{proof}

\begin{remark}
The similar result holds for $d_{2}$. 
\end{remark}

\begin{corollary}
\label{symmetry gap}
Suppose that $\L=L_{1}\cup L_{2}$ is an L--space link with $b_{1}=s_{1}$ for some maximal lattice point $(s_{1}, s_{2})\in \H(\L)$, and the coefficient of $t_{1}^{-s_{1}-1/2}t_{2}^{s_{2}+1/2}$ in the symmetrized Alexander polynomial $\Delta_{\L}(t_{1}, t_{2})$ is also nonzero. If $S^{3}_{d_{1}, d_{2}}(\L)$ is an L--space with $d_{2}>2b_{2}$, then $d_{1}>2b_{1}$. 
\end{corollary}

\begin{proof}
It is straight-forward from Proposition \ref{symmetry}. 

\end{proof}

\begin{corollary}
\label{if and only if2}
Let $\L=L_{1}\cup L_{2}$ be an L--space link with $b_{1}=s_{1}$ and $b_{2}=s'_{2}$ for some maximal lattice points  $(s_{1}, s_{2})$ and $(s'_{1}, s'_{2})$.  Suppose that the coefficients of $t_{1}^{-s_{1}-1/2}t_{2}^{s_{2}+1/2}$ and $t_{1}^{s'_{1}+1/2}t_{2}^{-s'_{2}-1/2}$ in the symmetrized Alexander polynomial $\Delta_{\L}(t_{1}, t_{2})$ are nonzero. Then $S^{3}_{d_{1}, d_{2}}(\L)$ is an L--space if and only if $d_{1}>2b_{1}$ and $d_{2}>2b_{2}$. 
\end{corollary}

\begin{proof}
The ``if " part is straightforward by large surgery formula. Now we prove the ``only if " part. Since there exist maximal lattice points for the link $\L$, it is type (A). By Lemma \ref{knot L--space}, if $S^{3}_{d_{1}, d_{2}}(\L)$ is an L--space, then either  $S^{3}_{d_{1}}(L_{1})$  or $S^{3}_{d_{2}}(L_{2})$ is an L--space. Without loss of generality, we assume that $S^{3}_{d_{1}}(L_{1})$ is an L--space. By the surgery induction (Lemma \ref{induction}), $S^{3}_{d_{1}, d_{2}+k}(\L)$ is an L--space for any $k>0$. So $S^{3}_{d_{1}, 2b_{2}+1}(\L)$ is an L--space. By Corollary \ref{symmetry gap}, $d_{1}>2b_{1}$. Similarly, we obtain that  $d_{2}>2b_{2}$ . Hence, $d_{1}> 2b_{1}$ and $d_{2}>2b_{2}$.

\end{proof}

\begin{example}

The mirror $\L=L_{1}\cup L_{2}$ of $L7a3$ is a 2-component L--space link with linking number 0, where $L_{1}$ is the right-handed trefoil and $L_{2}$ is the unknot \cite{Liu}. Its Alexander polynomial equals
$$\Delta_{\L}(t_{1}, t_{2})=-(t_{1}^{1/2}-t_{1}^{-1/2})(t_{2}^{1/2}-t_{2}^{-1/2})(t_{2}+t_{2}^{-1}). $$
In \cite[Example 4.4]{Liu2}, we compute its $H$-function, and get  $b_{1}=0$ and $b_{2}=1$. Its Alexander polynomial satisfies the assumptions in Corollary \ref{if and only if2}. Hence $S^{3}_{d_{1}, d_{2}}(\L)$ is an L--space if and only if $d_{1}>0, d_{2}>2$. 
\end{example}

In the rest of the section, we consider cables on links $\L$ which  satisfy the assumptions in Corollary \ref{if and only if2}. 
Let $\L=L_{1}\cup \cdots L_{n}\subset S^{3}$ be an L--space link. Let $p, q$ be the coprime positive integers. The link $\L_{p, q}=L_{(p, q)}\cup L_{2}\cup \cdots \cup L_{n}$ is an L--space link if $q/p$ is sufficiently large \cite[Proposition 2.8]{BG}. 

Given coprime positive integers $p, q$, define the map $T: \mathbb{R} \rightarrow \mathbb{R}$ as:
$$T(s)=ps+(p-1)(q-1)/2.$$

\begin{lemma}
\label{tranb}
Let $\L=L_{1}\cup L_{2}$ denote an L--space link, and   $\L_{p, q} $ denote its cable link where $p, q$ are coprime positive integers with $q/p$ sufficiently large. Then   $b_{2}(\L_{p, q})=b_{2}(\L)$ and 
\begin{equation}
\label{cableb}
b_{1}(\L_{p, q}) = \left\{
        \begin{array}{ll}
           T(b_{1}(\L)+1)-1 & \quad \textup{ if } \L \textup{ is type (B) or it is type (A) and } l \textup{ is even, } \\
           T(b_{1}(\L)+1/2)-1/2 & \quad  \textup{ if } \L \textup{ is type (A) and } l \textup{ is odd. }
        \end{array}
    \right. 
\end{equation}
\end{lemma}

\begin{proof}
By Lemma \ref{boundb}, $b_{1}(\L)\geq g(L_{1})-1+l/2$. We first suppose that $b_{1}(\L)\geq g(L_{1})+l/2$. By a similar argument to the one in Corollary \ref{positive surgery 2}, there exists a lattice point $(s_{1}, s_{2})\in \H(\L)$ such that 
$$H(s_{1}, s_{2})=a+1, H(s_{1}+1, s_{2})=H(s_{1}+1, s_{2}+1)=H(s_{1}+1, s_{2})=a$$
where $a\geq 0$, and $s_{1}=b_{1}$ or $b_{1}-1/2$ depending on the parity of the linking number. Then $\chi(HFL^{-}(s_{1}+1, s_{2}+1))=-1$, and it is the coefficient of the term $t_{1}^{s_{1}+1}t_{2}^{s_{2}+1}$ in $\tilde{\Delta}_{\L}(t_{1}, t_{2})$. By the definition of $b_{1}$, it is also not hard to see that the coefficients of $t_{1}^{y_{1}}t_{2}^{y_{2}}$ are $0$ in $\tilde{\Delta}_{\L}(t_{1}, t_{2})$ for all  $(y_{1}, y_{2})\succ (s_{1}+1,  s_{2}+1)$. 

Recall that the Alexander polynomial of the cable link $\L_{p, q}$ is computed by Turaev in \cite[Theorem 1.3.1]{Tu}, 
\begin{equation}
\Delta_{\L_{p, q}}(t_{1}, t_{2})=\Delta_{\L}(t_{1}^{p}, t_{2})\dfrac{t_{1}^{pq/2}-t_{1}^{-pq/2}}{t_{1}^{q/2}-t_{1}^{-q/2}}.
\end{equation}
Then 
\begin{equation}
\label{cableal2}
\tilde{\Delta}_{\L_{p, q}}(t_{1}, t_{2})=t_{1}^{1/2-p/2}\tilde{\Delta}_{\L}(t_{1}^{p}, t_{2})\dfrac{t_{1}^{pq/2}-t_{1}^{-pq/2}}{t_{1}^{q/2}-t_{1}^{-q/2}}. 
\end{equation}
Here $\dfrac{t_{1}^{pq/2}-t_{1}^{-pq/2}}{t_{1}^{q/2}-t_{1}^{-q/2}}$ is a Laurent polynomial of degree $pq/2-q/2$. 

Observe that $T(s)=ps+(1/2-p/2)+(pq-q)/2$. We claim that the coefficients of $t_{1}^{y_{1}}t_{2}^{y_{2}}$ in $\tilde{\Delta}_{\L}(t_{1}, t_{2})$ are $0$ for all $\bm{y}\succ \bm{s}$ if and only if for all $\bm{y}'\succ (T(s_{1}), s_{2})$, the coefficients of $t_{1}^{y'_{1}}t_{2}^{y'_{2}}$ are $0$ in $\tilde{\Delta}_{\L_{p, q}}(t_{1}, t_{2})$. The proof can be found in \cite[Theorem 4.7]{Liu2}, and  we don't  repeat the argument here. By the claim, the coefficients of the terms $t_{1}^{y'_{1}}t_{2}^{y'_{2}}$ in $\tilde{\Delta}_{\L_{p, q}}(t_{1}, t_{2})$ are  $0$ for all $y'_{1}>T(s_{1}+1)$ and the coefficient of the term $t_{1}^{T(s_{1}+1)}t_{2}^{s_{2}+1}$ is $-1$. So $b_{1}(\L_{p, q})=\lceil T(s_{1}+1)-1\rceil$. Replacing $s_{1}$ by $b_{1}$ or $b_{1}-1/2$, we prove \eqref{cableb}.

Now we assume that $b_{1}(\L)=g(L_{1})-1+l/2$. We claim that $b_{1}(\L_{p, q})=g(L_{p, q})-1+lp/2$. Note that  the linking number of $\L_{p, q}$ is $pl$. Recall that  the Alexander polynomial of the cable knot $L_{p, q}$ is computed by Turaev in \cite[Theorem 1.3.1]{Tu}, 
\begin{equation}
\label{cabelalex}
\Delta_{L_{p, q}}(t)=\dfrac{\Delta_{L_{1}}(t^{p})(t^{1/2}-t^{-1/2})}{t^{p/2}-t^{-p/2}}\cdot \dfrac{t^{pq/2}-t^{-pq/2}}{t^{p/2}-t^{-q/2}}.
\end{equation}
Here we are multiplying $\Delta_{L_{1}}(t^{p})$ by a Laurent polynomial of degree $T(0)$. 

It is not hard to see that for any L--space knot $K$, $g(K)$ is the top degree of the symmetrized Alexander polynomial of $K$ from the proof of Lemma \ref{genus0}. Then the monomial with the highest degree term in $\Delta_{L_{1}}(t)$ is $t^{g(L_{1})}$. By \eqref{cabelalex}, the highest degree term in $\Delta_{L_{p, q}}(t)$ is $t^{g(L_{1})p+(p-1)(q-1)/2}=T(g(L_{1}))$. 
Since $b_{1}(\L)=g(L_{1})-1+l/2$, the coefficients of the terms $t_{1}^{s_{1}}t_{2}^{s_{2}}$ in $\tilde{\Delta}_{\L}(t_{1}, t_{2})$ are $0$ for all $s_{1}>g(L_{1})+l/2$. By \eqref{cableal2}, the coefficients of terms $t_{1}^{y'_{1}}t_{2}^{y'_{2}}$ in $\tilde{\Delta}_{\L_{p, q}}(t_{1}, t_{2})$ are  $0$ for all $y'_{1}>T(g(L_{1})+l/2)=T(g(L_{1}))+pl/2$. This indicates that $b_{1}(\L_{p, q})=T(g(L_{1}))+lp/2-1$. Note that $b_{1}(\L)=g(L_{1})-1+l/2$. It is easy to see that $b_{1}(\L_{p, q})=T(b_{1}(\L)+1)-1$.

\end{proof}

\begin{lemma}
\label{mama}
Let $\L$ be a 2-component L--space link. The maximal lattice points $(s_{1}, s_{2})\in \H(\L)$ are one-to-one correspondence to the maximal lattice points $(T(s_{1}+1)-1, s_{2})\in \H(\L_{p, q})$ of the cable link $\L_{p, q}$. 

\end{lemma}

\begin{proof}
The maximal lattice point $(s_{1}, s_{2})$ has the property that $H(s_{1}, s_{2})=1$ and $H(s_{1}+1, s_{2})=H(s_{1}, s_{2}+1)=H(s_{1}+1, s_{2}+1)=0$. It is not hard to check that a lattice point $(s_{1}, s_{2})\in \H(\L)$ is maximal if and only if $s_{i}\geq g_{i}+l/2$ for $i=1, 2$, $\chi(HFL^{-}(s_{1}+1, s_{2}+1))=-1$ and $\chi(HFL^{-}(y_{1}, y_{2}))=0$ for all $(y_{1}, y_{2})\succ (s_{1}+1, s_{2}+1)$.

We first prove that if $(s_{1}, s_{2})\in \H(\L)$ is a maximal lattice point, then $(T(s_{1}+1)-1, s_{2})\in \H(\L_{cab})$ is a maximal lattice point. 
In the proof of Lemma \ref{tranb}, we see that $g(L_{p, q})=T(g(L_{1}))$, and $\chi(HFL^{-}(y'_{1}, y'_{2}))=0$ for all $(y'_{1}, y'_{2})\succ (T(s_{1}+1), s_{2}+1)$. Note that $T(s_{1}+1)-1=ps_{1}+p+(p-1)(q-1)/2-1$. Since $s_{1}\geq g(L_{1})+l/2$ and $p>1$, $T(s_{1}+1)-1\geq T(g(L_{1}))+lp/2$. By \eqref{cableal2}, $\chi(HFL^{-}(T(s_{1}+1), s_{2}+1))=-1$. Hence $(T(s_{1}+1)-1, s_{2})\in \H(\L_{p, q})$ is a maximal lattice point if $(s_{1}, s_{2})\in \H(\L)$ is a maximal lattice point. The converse also holds. If $(s'_{1}, s'_{2})\in \H(\L_{p, q})$ is a maximal lattice point, by \eqref{cableal2}, there exists a lattice point $(s_{1}, s_{2})\in \H(\L)$ such that $T(s_{1}+1)-1=s'_{1}$. By a similar argument, $\chi(HFL^{-}(s_{1}+1, s_{2}+1))=-1$ and $\chi(HFL^{-}(y_{1}, y_{2}))=0$ for all $(y_{1}, y_{2})\succ (s_{1}+1, s_{2}+1)$. It suffices to prove that $s_{1}\geq g(L_{1})+l/2$. Since $T(s_{1}+1)-1\geq T(g(L_{1}))+lp/2$, we have $s_{1}\geq g(L_{1})+l/2-1$. If $s_{1}=g(L_{1})-1+l/2$, either $$H(s_{1}, s_{2})=H(s_{1}+1, s_{2})=1, H(s_{1}+1, s_{2})=H(s_{1}+1, s_{2}+1)=0,$$
or there exists a maximal lattice point $(y_{1}, y_{2})\in \H(\L)$ such that $(y_{1}, y_{2})\succ (s_{1}, s_{2})$. In the former case, $\chi(HFL^{-}(s_{1}+1, s_{2}+1))=0$, which contradicts to the property $\chi(HFL^{-}(s_{1}+1, s_{2}+1))=-1$. The latter case contradicts to the property that  $\chi(HFL^{-}(y_{1}, y_{2}))=0$ for all $(y_{1}, y_{2})\succ (s_{1}+1, s_{2}+1)$. Hence, $(s_{1}, s_{2})\in \H(\L)$ is a maximal lattice point of $\L$.

\end{proof}

\begin{lemma}
\label{cablech}
Suppose that $\L=L_{1}\cup L_{2}$ is an L--space link with a maximal lattice point $\bm{s}=(s_{1}, s_{2})\in \H(\L)$ such that the coefficient of  $t_{1}^{-s_{1}-1/2}t_{2}^{ s_{2}+1/2}$   in  $\Delta_{\L}(t_{1}, t_{2})$ is nonzero. Then the coefficient of $t_{1}^{-T(s_{1}+1)+1/2, s_{2}+1/2}$ in $\Delta_{\L_{p, q}}(t_{1}, t_{2})$ is also nonzero corresponding the maximal lattice point $(T(s_{1}+1)-1, s_{2})\in \H(\L_{p, q})$.

\end{lemma}

\begin{proof}
By Lemma \ref{mama}, $(T(s_{1}+1)-1, s_{2})\in \H(\L_{p, q})$ is a maximal lattice point of $\L_{p, q}$. It suffices to prove that the coefficient of  $t^{-T(s_{1}+1)+1, s_{2}+1}$ is nonzero in $\tilde{\Delta}_{\L_{p, q}}(t_{1}, t_{2})$. Note that \eqref{cableal2} can be written as:
$$\tilde{\Delta}_{\L_{p, q}}(t_{1}, t_{2})=t_{1}^{1/2-p/2}\tilde{\Delta}_{\L}(t_{1}^{p}, t_{2}) (t_{1}^{q(p-1)/2}+t_{1}^{q(p-3)/2}+\cdots +t_{1}^{-q(p-1)/2}).$$
Then 
$$-T(s_{1}+1)+1=-ps_{1}-p-(p-1)(q-1)/2+1=1/2-p/2-ps_{1}-q(p-1)/2.$$
So we need to check that the term is not cancelled in $\tilde{\Delta}_{\L_{p, q}}(t_{1}, t_{2})$. Assume that there exists a term $t_{1}^{x}t_{2}^{y}$ in $\tilde{\Delta}_{\L}(t_{1}, t_{2})$ such that 
$$1/2-p/2+px+q(p-k)/2=1/2-p/2-ps_{1}-q(p-1)/2$$
for some $k\in \{ 1, 3, 5, \cdots, 2p-1\}$. Simplifying this equation, we get 
\begin{equation}
\label{cance}
(x+s_{1})p=-q(p-k)/2-q(p-1)/2=-q(p-l-1)
\end{equation}
where $k=2l+1$ and $l\in \{0, 1, \cdots, p-1\}$. Note  that $p$ and $q$ are coprime, and $0\leq p-l-1\leq p-1$. So the only solution to \eqref{cance} is $x=-s_{1}$. Hence, the term $t_{1}^{-T(s_{1}+1)+1}t_{2}^{s_{2}+1}$ has nonzero coefficient in $\tilde{\Delta}_{\L_{p, q}}(t_{1}, t_{2})$. 

\end{proof}

\begin{theorem}
\label{if and only 3}
Let $\L=L_{1}\cup L_{2}$ be an L--space link with $b_{1}=s_{1}$ and $b_{2}=s'_{2}$ for some maximal lattice points  $(s_{1}, s_{2})$ and $(s'_{1}, s'_{2})$.  Suppose that the coefficients of $t_{1}^{-s_{1}-1/2}t_{2}^{s_{2}+1/2}$ and $t_{1}^{s'_{1}+1/2}t_{2}^{-s'_{2}-1/2}$ in the symmetrized Alexander polynomial $\Delta_{\L}(t_{1}, t_{2})$ are nonzero. Then $S^{3}_{d_{1}, d_{2}}(\L_{p, q})$ is an L--space if and only if $d_{1}>2b_{1}(\L_{p, q}), d_{2}>2b_{2}(\L_{p, q})$ for all cable link $\L_{p, q}$.
\end{theorem}

\begin{proof}
This is straightforward from  Lemma \ref{cablech} and Corollary \ref{if and only if2}. 
\end{proof}

\indent

{\bf Proof of Corollary \ref{whitehead}: } The Whitehead link $Wh$ is an L--space link with vanishing linking number. Its Alexander polynomial is as follows:
$$\Delta(t_{1}, t_{2})=-(t_{1}^{1/2}-t_{1}^{-1/2})(t_{2}^{1/2}-t_{2}^{-1/2}).$$
By \eqref{Alex} and \eqref{computeh}, the $H$-function has the following values:

\begin{center}
\begin{tikzpicture}
\draw (1,0)--(1,5);
\draw (2,0)--(2,5);
\draw (3,0)--(3,5);
\draw (4,0)--(4,5);
\draw (0,1)--(5,1);
\draw (0,2)--(5,2);
\draw (0,3)--(5,3);
\draw (0,4)--(5,4);
\draw (0.5,4.5) node {2};
\draw (1.5,4.5) node {1};
\draw (2.5,4.5) node {0};
\draw (3.5,4.5) node {0};
\draw (4.5,4.5) node {0};
\draw (0.5,3.5) node {2};
\draw (1.5,3.5) node {1};
\draw (2.5,3.5) node {0};
\draw (3.5,3.5) node {0};
\draw (4.5,3.5) node {0};
\draw (0.5,2.5) node {2};
\draw (1.5,2.5) node {1};
\draw (2.5,2.5) node {1};
\draw (3.5,2.5) node {0};
\draw (4.5,2.5) node {0};
\draw (0.5,1.5) node {3};
\draw (1.5,1.5) node {2};
\draw (2.5,1.5) node {1};
\draw (3.5,1.5) node {1};
\draw (4.5,1.5) node {1};
\draw (0.5,0.5) node {4};
\draw (1.5,0.5) node {3};
\draw (2.5,0.5) node {2};
\draw (3.5,0.5) node {2};
\draw (4.5,0.5) node {2};
\draw [->,dotted] (0,2.5)--(5,2.5);
\draw [->,dotted] (2.5,0)--(2.5,5);
\draw (5,2.7) node {$s_1$};
\draw (2.3,5) node {$s_2$};
\end{tikzpicture}
\end{center}

So there is a maximal lattice point $(0, 0)\in \H(Wh)$ and $b_{1}(Wh)=b_{2}(Wh)=0$. It is easy to check that the Whitehead link satisfies the assumptions in Theorem \ref{if and only 3}. By Theorem \ref{if and only 3}, $S^{3}_{d_{1}, d_{2}}(Wh_{cab})$ is an L--space if and only if $d_{1}>2b_{1}(Wh_{cab})$ and $d_{2}>2b_{2}(Wh_{cab})$.  By \eqref{cableb}, $b_{i}(Wh_{cab})=p_{i}+(p_{i}-1)(q_{i}-1)/2-1$ for $i=1, 2$. Hence  $S^{3}_{d_{1}, d_{2}}(Wh_{cab})$ is an L--space if and only if $d_{i}>p_{i}q_{i}+p_{i}-q_{i}-1$ for $i=1,2$.


\begin{thebibliography}{99}
\bibitem{BG} M. Borodzik, E. Gorsky, \emph{Immersed concordances of links and Heegaard Floer homology}, Indiana Univ. Math. J. 67 (2018), no. 3, 1039-1083. 

\bibitem{BGW} S. Boyer, C. M. Gordon, and L. Watson, \emph{On L-spaces and left-orderable fundamental groups. } Mathematische Annalen, 356(4):1213-1245, 2013. 


\bibitem{GLM} E. Gorsky, B. Liu and A. Moore, \emph{Surgery on links of linking number zero and the Heegaard Floer d-invariant.} Preprint 2018, arXiv: 1810.10178, 2018. 

\bibitem{GN15} E. Gorsky, A. N\'emethi, \emph{Lattice and Heegaard Floer homologies of algebraic links.} Int. Math. Res. Not. IMRN, (23): 12737-12780, 2015. 

\bibitem{GN} E. Gorsky, A. N\'emethi, \emph{On the set of L--space surgeries for links. }  Advances in Mathematics 333 (2018), 386-422. 

\bibitem{HRRW} J. Hanselman, J. Rasmussen, S. D. Rasmussen, and L. Watson. \emph{Taut foliations on graph manifolds}. Preprint 2015, arXiv: 1508.05911, 2015. 

\bibitem{HRW} J. Hanselman, J. Rasmussen, and L. Watson. \emph{Bordered Floer homology for manifolds with torus boundary via immersed curves}. Preprint 2016, arXiv: 1604.03466. 

\bibitem{Kid} M. E. Kidwell. \emph{Alexander polynomials of links of small order}. Illinois J. Math. 22 (1978), no. 3, 459-475.

\bibitem{Liu1} B. Liu, \emph{Heegaard Floer homology of $L$-space links with two components}. Pacific J. Math. 298 (2019), no.1, 83-112. 

\bibitem{Liu2} B. Liu, \emph{Four genera of links and Heegaard Floer homology}. To appear in Algebr. Geom. Topol., arXiv: 1805.02122.


 \bibitem{Liu} Y. Liu, \emph{$L$-space surgeries on links,} Quantum Topol. {\bf 8} (2017), no. 3, 505-570. 
 
 \bibitem{MO} C. Manolescu and P. S.  Ozsv\'ath. \emph{Heegaard Floer homology and integer surgeries on links}. arXiv: 1011.1317v4, 2010. 
 
 \bibitem{Ne} A. N\'emethi. \emph{Lattice cohomology of normal surface singularities.} Publ. Res. Inst. Math. Sci., 44:507-543, 2008. 
 
 \bibitem{OS3} P. S. Ozsv\'ath, Z. Szab\'o, \emph{Holomorphic disks and genus bounds.} Geom. Topol. 8 (2004), 311-334. 
 
 \bibitem{OS04b} P. S. Ozsv\'ath, Z. Szab\'o, \emph{Holomorphic disks and knot invariants.} Adv. Math., 186(1):58-116, 2004. 
 
 \bibitem{OS2} P. S. Ozsv\'ath, Z. Szab\'o, \emph{Holomorphic disks and topological invariants for closed three-manifolds.} Ann. of Math. (2) 159 (2004), no. 3, 1027-1158. 

\bibitem{OS1} P. S. Ozsv\'ath, Z. Szab\'o, \emph{Holomorphic disks and three-manifold invariants: Properties and applications,} Annal of Mathematics, {\bf 159} (2004), 1159-1245. 

\bibitem{OS05} P. S. Ozsv\'ath, Z. Szab\'o, \emph{On knot Floer homology and lens space surgeries.} Topology, 44(6):1281-1300, 2005. 

\bibitem{OS11} P. S. Ozsv\'ath, Z. Szab\'o, \emph{Knot Floer homology and rational surgeries}. Algebr. Geom. Topol., 11(1): 1-68, 2011. 

\bibitem{OS08} P. S. Ozsv\'ath, Z. Szab\'o, \emph{Holomorphic disks, link invariants and the multi-variable Alexander polynomial.} Algebr. Geom. Topol., 8(2):615-692, 2008.  

\bibitem{Ras17} S. D. Rasmussen, \emph{L--space intervals for graph manifolds and cables.} Compos. Math. 153 (2017) no. 5, 1008-1049. 

\bibitem{Ras} S. D. Rasmussen, \emph{L--space surgeries on satellites by algebraic links.} Preprint 2017, arXiv: 1703.06874, 2017. 

\bibitem{Tu} V. Turaev, \emph{Reidemeister torsion in knot theory. } Russian Math. Surveys 41(1986), No. 1, 119-182. 

\end{thebibliography}
\end{document}